\documentclass{amsart}
\usepackage{amssymb, latexsym, MnSymbol,graphicx,psfrag}

\usepackage{color}

\newcommand{\Ga}{\Gamma}
\newcommand{\Si}{\Sigma}

\newcommand{\bC}{\mathbb{C}}
\newcommand{\bE}{\mathbb{E}}
\newcommand{\bP}{\mathbb{P}}

\newcommand{\bZ}{\mathbb{Z}}
\newcommand{\balpha}{{\boldsymbol\alpha}}

\newcommand{\cH}{\mathcal{H}}
\newcommand{\cE}{\mathcal{E}}
\newcommand{\cL}{\mathcal{L}}
\newcommand{\cM}{\mathcal{M}}
\newcommand{\cO}{\mathcal{O}}
\newcommand{\cS}{\mathcal{S}}

\newcommand{\ev}{\mathrm{ev}}
\newcommand{\Res}{\mathrm{Res}}
\newcommand{\val}{ {\mathrm{val}} }
\newcommand{\Aut}{\mathrm{Aut}}
\newcommand{\vir}{{\mathrm{vir}} }
\newcommand{\Jac}{{\mathrm{Jac}} }
\newcommand{\ch}{\mathrm{ch}}
\newcommand{\SYZ}{\mathrm{SYZ}}

\newcommand{\bu}{\mathbf{u}}
\newcommand{\bt}{\mathbf{t}}
\newcommand{\one}{\mathbf{1}}

\newcommand{\w}{\mathsf{w}}

\newcommand{\halpha}{{\hat\alpha}}
\newcommand{\hbeta}{{\hat\beta}}

\newcommand{\hualpha}{{\hat{\underline \alpha}}}
\newcommand{\hubeta}{{\hat{\underline \beta}}}
\newcommand{\hugamma}{{\hat{\underline \gamma}}}

\newcommand{\tS}{\widetilde{S}}
\newcommand{\tch}{\widetilde{\ch}}

\newcommand{\tW}{\widetilde{W}}
\newcommand{\txi}{\widetilde{\xi}}
\newcommand{\tx}{\widetilde{x}}

\newcommand{\vGa}{\vec{\Gamma}}
\newcommand{\bGa}{\mathbf{\Gamma}}
\newcommand{\Mbar}{\overline{\cM}}

\newcommand{\lp}{{(\!(}}
\newcommand{\rp}{{)\!)}}

\newtheorem{Theorem}{Theorem}

\newtheorem{dummy}{dummy}[section]
\newtheorem{lemma}[dummy]{Lemma}
\newtheorem{theorem}[dummy]{Theorem}

\newtheorem{corollary}[dummy]{Corollary}
\newtheorem{proposition}[dummy]{Proposition}
\newtheorem{definition}[dummy]{Definition}
\newtheorem{remark}[dummy]{Remark}

\begin{document}
\title[Eynard-Orantin Recursion and Equivariant mirror symmetry for $\bP^1$]{The Eynard-Orantin
Recursion and Equivariant mirror symmetry for the projective line}
\author{Bohan Fang}
\address{Bohan Fang, Beijing International Center for Mathematical Research, Peking University, 5 Yiheyuan Road, Beijing 100871, China}
\email{bohanfang@gmail.com  }

\author{Chiu-Chu Melissa Liu}
\address{Chiu-Chu Melissa Liu, Department of Mathematics, Columbia University, 2990 Broadway, New York, NY 10027}
\email{ccliu@math.columbia.edu}

\author{Zhengyu Zong}
\address{Zhengyu Zong, Yau Mathematical Sciences Center, Tsinghua University, Jin Chun Yuan West Building,
Tsinghua University, Haidian District, Beijing 100084, China}
\email{zyzong@mail.tsinghua.edu.cn}

\begin{abstract}
We study the equivariantly perturbed mirror Landau-Ginzburg model of $\bP^1$. We show that the Eynard-Orantin recursion on this model encodes all genus all descendants equivariant Gromov-Witten invariants of $\bP^1$. The non-equivariant limit of this result is the Norbury-Scott conjecture \cite{NS, DOSS}, while by taking large radius limit we recover
the Bouchard-Mari\~no conjecture on simple Hurwitz numbers \cite{BM08}.
\end{abstract}

\maketitle

\tableofcontents

\section{Introduction}

The equivariant Gromov-Witten theory of $\bP^1$ has been studied extensively.
In \cite{OP1, OP2}, Okounkov-Pandharipande completely solved the equivariant Gromov-Witten
theory of the projective line and  established a GW/H correspondence between the stationary sector of
Gromov-Witten theory of $\bP^1$ and Hurwitz theory. In \cite{G01}, Givental derived a quantization
formula for all genus descendant potential of the equivariant Gromov-Witten
theory of $\bP^1$ (and more generally, $\bP^n$). In the non-equivariant limit, these results imply
the Virasoro conjecture of $\bP^1$.

The Norbury-Scott conjecture \cite{NS} relates (non-equivariant) Gromov-Witten invariants of
$\bP^1$  to Eynard-Orantin invariants \cite{EO07} of the affine plane curve
$\{ x = Y+\frac{1}{Y}:(x,Y)\in \bC\times\bC^*\}$. In \cite{DOSS}, P. Dunin-Barkowski, N. Orantin,
S. Shadrin, and L. Spitz relate the Eynard-Orantin topological recursion to the Givental formula
for the ancestor formal Gromov-Witten potential, and prove the Norbury-Scott conjecture using their main result
and Givental's quantization formula for all genus descendant potential of the (non-equivariant)
Gromov-Witten theory of $\bP^1$. It is natural to ask if the Norbury-Scott conjecture
can be extended to the equivariant setting, such that the original conjecture can be
recovered in the non-equivariant limit.

\subsection{Main Results}
Our first main result  (Theorem \ref{main} in Section \ref{sec:main-results})
relates equivariant Gromov-Witten invariants of $\bP^1$ to
the Eynard-Orantin invariants \cite{EO07} of the affine curve
$$
\{ x = t^0+ Y +\frac{Q e^{t^1}}{Y} +\w_1 \log Y +\w_2\log\frac{Qe^{t^1}}{Y}): (x,Y)\in \bC\times \bC^*\}
$$
where $t^0, t^1$ are complex parameters,  $\w_1,\w_2$ are equivariant parameters of the torus $T=(\bC^*)^2$
acting on $\bP^1$, and $Q$ is the Novikov variable encoding the degree of the stable maps
to $\bP^1$ (see Section \ref{sec:TGW}).
The superpotential of the $T$-equivariant Landau-Ginzburg mirror of the projective line is given by
$$
W_t^\w:\bC^*\to \bC,\quad W_t^\w(Y)= t_0+ Y +\frac{Qe^{t^1}}{Y} +\w_1 \log Y +\w_2\log\frac{Qe^{t^1}}{Y},
$$
so Theorem \ref{main} can be viewed as a version of
all genus equivariant mirror symmetry for $\bP^1$. We prove Theorem \ref{main} using the main result
in \cite{DOSS} and a suitable version of Givental's formula for all genus {\em equivariant} descendant
Gromov-Witten potential of $\bP^n$ \cite{G01} (see also \cite{LP}).

Our second main result (Theorem \ref{full-descendant} in Section \ref{sec:main-results})
gives a precise correspondence between (A) genus-$g$, $n$-point descendant equivariant Gromov-Witten invariants of $\bP^1$,
and (B) Laplace transforms of the Eynard-Orantin invariant $\omega_{g,n}$ along Lefschetz thimbles.
This result generalizes the known relation between the A-model genus-$0$ $1$-point descendant Gromov-Witten invariants
and the B-model oscillatory integrals.

\subsection{Non-equivariant limit and the Norbury-Scott conjecture}
Taking the non-equivariant limit $\w_1=\w_2=0$, we obtain
$$
W_t(Y) = t^0 +  Y+ \frac{Qe^{t^1}}{Y}.
$$
which is the superpotential of the (non-equivariant) Landau-Ginzburg mirror
for the projective line. We obtain all genus (non-equivariant) mirror symmetry
for the projective line.

In the stationary phase $t^0=t^1=0,Q=1$, the curve becomes
$$
\{ x = Y+ Y^{-1}:(x,Y)\in \bC\times \bC^*\},
$$
and Theorem \ref{main} specializes to the Norbury-Scott conjecture \cite{NS}.
(See Section \ref{sec:NS} for details.)

\subsection{Large radius limit and the Bouchard-Mari\~{n}o conjecture}
Let $\w_2=0$, $t_0=0$ and $q=Qe^{t^1}$, we obtain
$$
x = Y +\frac{q}{Y}+ \w_1 \log Y
$$
which reduces to
$$
x= Y +\w_1\log Y
$$
in the large radius limit $q\to 0$. The $\bC^*$-equivariant
mirror of the affine line $\bC$ is given by
$$
W:\bC^*\to \bC,\quad W(Y)= Y + \w_1 \log Y.
$$
In the large radius limit, we obtain a version
of all genus $\bC^*$-equivariant mirror symmetry
of the affine line $\bC$.

In particular, let $\w_1=-1$ and $X=e^{-x}$, we
obtain the Lambert curve
$$
X = Y e^{-Y}.
$$
In this limit, Theorem \ref{main} specializes
to the Bouchard-Mari\~{n}o conjecture \cite{BM08} relating
simple Hurwitz numbers (related to linear Hodge integrals
by the ELSV formula \cite{ELSV,GV}) to Eynard-Orantin invariants
of the Lambert curve.  (See Section \ref{sec:BM} for details.)

In \cite{BEMS}, Borot-Eynard-Mulase-Safnuk introduced a new matrix model
representation for the generating function of simple Hurwitz
numbers, and proved the Bouchard-Mari\~{n}o conjecture.
In \cite{EMS}, Eynard-Mulase-Safnuk provided another proof of the
Bouchard-Mari\~{n}o conjecture  using the cut-and-joint equation
of simple Hurwitz numbers. Recently, a new proof of the ELSV formula
and a new proof of the Bouchard-Mari\~{n}o conjecture are given in  \cite{DKOSS}.

\subsection*{Acknowledgment}
We thank P. Dunin-Barkowski, B. Eynard, M. Mulase, P. Norbury, and N. Orantin for helpful conversations.
The research of the authors is partially supported by NSF DMS-1206667 and NSF DMS-1159416.

\section{A-model}
Let $T=(\bC^*)^2$ act on $\bP^1$ by
$$
(t_1,t_2)\cdot [z_1, z_2]=
[t_1^{-1} z_1, t_2^{-1} z_2].
$$
Let $\bC[\w]:=\bC[\w_1,\w_2]$ be
the $T$-equivariant cohomology of a point:
$H_T^*(\mathrm{point};\bC)=\bC[\w]$.

\subsection{Equivariant cohomology of $\bP^1$}
The $T$-equivariant cohomology of $\bP^1$ is given by
$$
H^*_T(\bP^1;\bC)= \bC[H,\w]/\langle (H-\w_1)(H-\w_2)\rangle
$$
where $\deg H=\deg \w_i=2$.
Let $p_1=[1,0]$ and $p_2=[0,1]$ be the $T$ fixed points.
Then $H|_{p_i}= \w_i$. The $T$-equivariant Poincar\'{e} dual
of $p_1$ and $p_2$ are $H-\w_2$ and $H-\w_1$, respectively.
Let
$$
\phi_1 := \frac{H-\w_2}{\w_1-\w_2},\quad
\phi_2 := \frac{H-\w_1}{\w_2-\w_1} \in H^*_T(\bP^1;\bC)\otimes_{\bC[\w]}\bC[\w, \frac{1}{\w_1-\w_2}]
$$
Then $\deg \phi_\alpha=0$,
$$
\phi_\alpha \cup \phi_\beta = \delta_{\alpha \beta}\phi_\alpha,
$$
So $\{\phi_1, \phi_2\}$ is a canonical basis of
the semisimple algebra
$$
H_T^*(\bP^1;\bC)\otimes_{\bC[\w]} \bC[\w,\frac{1}{\w_1-\w_2}].
$$
We have
$$
\phi_1+\phi_2=1,
$$
$$
(\phi_\alpha, \phi_\beta):=\int_{\bP^1}\phi_\alpha \cup \phi_\beta
=\delta_{\alpha \beta} \int_{\bP^1}\phi_\alpha = \frac{\delta_{\alpha \beta}}{\Delta^\alpha},\qquad\alpha,\beta\in\{1,2\},
$$
where
$$
\Delta^1 =\w_1-\w_2 ,\quad \Delta^2 =\w_2-\w_1.
$$

Cup product with the hyperplane class is given by
$$
H\cup  \phi_\alpha = \w_\alpha \phi_\alpha,\quad \alpha=1, 2.
$$

\subsection{Equivariant Gromov-Witten invariants of $\bP^1$} \label{sec:TGW}
Suppose that $d>0$ or $2g-2+n>0$, so that $\Mbar_{g,n}(\bP^1,d)$ is non-empty.
Given $\gamma_1,\ldots, \gamma_n\in H_T^*(\bP^1,\bC)$ and $a_1,\ldots,a_n\in \bZ_{\geq 0}$, we
define genus $g$, degree $d$, $T$-equivariant descendant Gromov-Witten invariants of $\bP^1$:
$$
\langle \tau_{a_1}(\gamma_1)\ldots\tau_{a_n}(\gamma_n)\rangle_{g,n,d}^{\bP^1,T}:=
\int_{[\Mbar_{g,n}(\bP^1,d)]^\vir} \prod_{j=1}^n \psi_j^{a_j}\ev_j^*(\gamma_j) \in \bC[\w]
$$
where $\ev_j: \Mbar_{g,n}(\bP^1,d) \to \bP^1$ is the evaluation at the $j$-th marked point,
which is a $T$-equivariant map.   We define genus $g$, degree $d$ primary Gromov-Witten invariants:
$$
\langle \gamma_1,\ldots, \gamma_n\rangle_{g,n,d}^{\bP^1,T}:=
\langle \tau_0(\gamma_1) \cdots \tau_0(\gamma_n)\rangle_{g,n,d}^{\bP^1,T}.
$$

Let $\bt=t^0 1 + t^1 H$. If $2g-2+n>0$, define
$$
\llangle \tau_{a_1}(\gamma_1),\ldots,\tau_{a_n}(\gamma_n)\rrangle_{g,n}^{\bP^1,T}
:= \sum_{d\geq 0} Q^d \sum_{\ell =0}^\infty \frac{1}{\ell !}\langle \tau_{a_1}(\gamma_1) \cdots \tau_{a_n}(\gamma_n)
\underbrace{\tau_0(\bt) \cdots \tau_0(\bt)}_{\ell \textup{ times}}\rangle_{g,n+\ell,d}^{\bP^1,T}
$$

Suppose that $2g-2+n+m>0$. Given $\gamma_1,\ldots, \gamma_{n+m}\in H^*_T(\bP^1)$, we define
\begin{eqnarray*}
&& \langle \frac{\gamma_1}{z_1-\psi_1},\ldots, \frac{\gamma_n}{z_n-\psi_n}, \gamma_{n+1},\ldots, \gamma_{n+m} \rangle^{\bP^1,T}_{g,n+m,d}\\
&:=&  \sum_{a_1,\ldots,a_n\geq 0} \langle \tau_{a_1}(\gamma_1)\cdots \tau_{a_n}(\gamma_n)
\tau_0(\gamma_{n+1})\cdots \tau_0(\gamma_{n+m})\rangle^{\bP^1,T}_{g,n+m,d}\prod_{j=1}^n z_j^{-a_j-1}.
\end{eqnarray*}
In particular, if $n+m\geq 3$ then
\begin{equation}\label{eqn:genus-zero-stable}
\begin{aligned}
  & \langle \frac{\gamma_1}{z_1-\psi_1},\ldots, \frac{\gamma_n}{z_n-\psi_n}, \gamma_{n+1},\ldots,\gamma_{n+m}\rangle^{\bP^1,T}_{0,n+m,0}\\
= & \frac{1}{z_1\cdots z_n}(\frac{1}{z_1}+\cdots + \frac{1}{z_n})^{n+m-3}\int_{\bP^1}\gamma_1\cup \cdots \cup \gamma_{n+m}
\end{aligned}
\end{equation}
where we use the fact $\Mbar_{0,n+m}(\bP^1,0)=\Mbar_{0,m+n}\times \bP^1$, and the identity
$$
\int_{\Mbar_{0,k}}\psi_1^{a_1}\cdots \psi_k^{a_k} =\begin{cases}
\displaystyle{ \frac{(k-3)!}{\prod_{j=1}^k a_j!} } & \textup{if } a_1+\cdots + a_k = k-3,\\
0, & \textup{otherwise}.
\end{cases}
$$
We use the second line of \eqref{eqn:genus-zero-stable} to extend the definition
of the correlator in the first line of \eqref{eqn:genus-zero-stable}
to the unstable cases $(n,m)=(1,0), (1,1), (2,0)$:
\begin{eqnarray*}
\langle\frac{\gamma_1}{z_1-\psi_1}\rangle^{\bP^1,T}_{0,1,0} &:=& z_1\int_{\bP^1}\gamma_1 \\
\langle\frac{\gamma_1}{z_1-\psi_1}, \gamma_2\rangle^{\bP^1,T}_{0,2,0} &:=& \int_{\bP^1}\gamma_1\cup \gamma_2\\
\langle\frac{\gamma_1}{z_1-\psi_1},\frac{\gamma_2}{z_2-\psi_2}\rangle^{\bP^1,T}_{0,2,0} &:=& \frac{1}{z_1+z_2} \int_{\bP^1}\gamma_1\cup \gamma_2
\end{eqnarray*}

Suppose that $2g-2+n+m>0$ or $n>0$. Given $\gamma_1,\ldots, \gamma_{n+m}\in H_T^*(\bP^1)$, we define
\begin{eqnarray*}
&& \llangle \frac{\gamma_1}{z_1-\psi_1},\ldots, \frac{\gamma_n}{z_n-\psi_n}, \gamma_{n+1},\ldots, \gamma_{n+m} \rrangle^{\bP^1,T}_{g,n+m} \\
&:= &\sum_{d\geq 0}\sum_{\ell\geq 0} \frac{Q^d}{\ell!}
\langle \frac{\gamma_1}{z_1-\psi_1},\ldots, \frac{\gamma_n}{z_n-\psi_n}, \gamma_{n+1},\ldots, \gamma_{n+m},
\underbrace{\bt,\ldots,\bt}_{\ell \textup{ times }} \rangle^{\bP^1,T}_{g,n+m+\ell,d}.
\end{eqnarray*}


Let $q=Qe^{t^1}$. Then for $m\geq 3$,
$$
\llangle \gamma_1,\ldots, \gamma_m\rrangle^{\bP^1,T}_{0,m} =
\sum_{d\geq 0} q^d \langle
\gamma_1,\ldots, \gamma_m \rangle^{\bP^1,T}_{0,m,d}
= \delta_{m,3}\int_{\bP^1}\gamma_1\cup \cdots\cup \gamma_m
+ q \prod_{i=1}^m (\int_{\bP^1} \gamma_i).
$$

\subsection{Equivariant quantum cohomology of $\bP^1$}\label{sec:QH}
As a $\bC[\w]$-module, $QH^*_T(\bP^1;\bC) = H^*_T(\bP^1;\bC)$. The
ring structure is given by the quantum product $*$ defined by
$$
(\gamma_1 \star \gamma_2, \gamma_3) = \llangle \gamma_1,\gamma_2,\gamma_3\rrangle_{0,3}^{\bP^1,T},
$$
or equivalently,
$$
\gamma_1\star \gamma_2 = \gamma_1 \cup \gamma_2 + q (\int_{\bP^1}\gamma_1)(\int_{\bP^1}\gamma_2).
$$
where $\cup$ is the product in $H^*_T(\bP^1)$, and $q=Qe^{t^1}$.
In particular,
$$
H\star H = (\w_1+\w_2)H -\w_1\w_2 + q.
$$
The $T$-equivariant quantum cohomology of $\bP^1$ is
$$
QH^*_T(\bP^1;\bC)=\bC[H,\w,q]/\langle (H-\w_1)\star(H-\w_2)-q\rangle
$$
where $\deg H=\deg \w_i=2$, $\deg q =4$.

The (non-equivariant) quantum cohomology of $\bP^1$ is
$$
\bC[H,q]/\langle H\star H -q\rangle
$$

Let
\begin{eqnarray*}
\phi_1(q) &=& \frac{1}{2}+ \frac{H-\frac{\w_1+\w_2}{2} }{ (\w_1-\w_2)\sqrt{ 1+\frac{4q}{(\w_1-\w_2)^2}} }, \\
\phi_2(q) &=& \frac{1}{2}+\frac{H-\frac{\w_1+\w_2}{2}}{(\w_2-\w_1)\sqrt{1+\frac{4q}{(\w_1-\w_2)^2}} }.
\end{eqnarray*}
Then
$$
\phi_\alpha(q)\star\phi_\beta(q)=\delta_{\alpha\beta} \phi_\alpha(q),
$$
So $\{\phi_1(q),\phi_2(q)\}$ is a canonical basis of the semi-simple algebra
$$
QH^*_T(\bP^1;\bC)\otimes\bC[\w,\frac{1}{\Delta^1(q)}]
$$
where $\Delta^1(q)$ is defined by \eqref{eqn:Delta-one}. We also have
\begin{eqnarray*}
( \phi_\alpha(q),\phi_\beta(q) ) &= & ( 1\star \phi_\alpha(q), \phi_\beta(q) )
= ( 1, \phi_\alpha(q)\star\phi_\beta(q) ) \\
&=& \delta_{\alpha\beta} ( 1, \phi_\alpha(q) )
=\delta_{\alpha\beta}\int_{\bP^1}\phi_\alpha(q) =\frac{\delta_{\alpha\beta}}{\Delta^\alpha(q)},
\end{eqnarray*}
where
\begin{eqnarray*}
\Delta^1(q) &=& (\w_1-\w_2)\sqrt{ 1+\frac{4q}{(\w_1-\w_2)^2} },  \label{eqn:Delta-one}\\
\Delta^2(q) &=&  (\w_2-\w_1)\sqrt{ 1+\frac{4q}{(\w_1-\w_2)^2} } =-\Delta^1(q).
\end{eqnarray*}
Quantum multiplication by the hyperplane class is given by
$$
H\star \phi_\alpha =  \frac{\w_1+\w_2 +\Delta^\alpha(q)}{2} \phi_\alpha,\quad \alpha=1,2.
$$

Finally, we take the non-equivariant limit $\w_2=0$, $\w_1\to 0^+$.
We obtain:
$$
\phi_1(q)=\frac{1}{2}+\frac{H}{2\sqrt{q}},
\quad \phi_2(q)=\frac{1}{2}-\frac{H}{2\sqrt{q}},
$$
$$
\Delta^1(q)=2\sqrt{q},\quad \Delta^2(q)=-2\sqrt{q},
$$
$$
H\star\phi_1(q)= \sqrt{q}\phi_1(q),\quad
H\star\phi_2(q)=-\sqrt{q}\phi_2(q).
$$
These non-equivariant limits coincide with the results in
\cite[Section 2]{SS}.
\subsection{The A-model canonical coordinates and the $\Psi$-matrix}\label{sec:A-canonical}
Let $\{t^0,t^1\}$ be the flat coordinates with respect to the basis $\{\one,H\}$, and
let $\{u^1,u^2\}$ be the canonical coordinates with respect to the basis $\{\phi_1(q), \phi_2(q)\}$.
Then
\begin{eqnarray*}
\frac{\partial}{\partial u^1}&=&\frac{1}{2}(1-\frac{\w_1+\w_2}{\Delta^1(q)})\frac{\partial}{\partial t^0}
+\frac{1}{\Delta^1(q)}\frac{\partial}{\partial t^1},\\
\frac{\partial}{\partial u^2}&=& \frac{1}{2}(1-\frac{\w_1+\w_2}{\Delta^2(q)}) \frac{\partial}{\partial t^0}
+\frac{1}{\Delta^2(q)}\frac{\partial}{\partial t^1},
\end{eqnarray*}
\begin{eqnarray*}
du^1&=& dt^0+\frac{1}{2}(\Delta^1(q)+ \w_1+\w_2)dt^1 ,\\
du^2&=& dt^0+\frac{1}{2}(\Delta^2(q) +\w_1+\w_2)dt^1.
\end{eqnarray*}
The above equations determine the canonical coordinates
$u^1$ and $u^2$ up to a constant in $\bC[\w_1,\w_2,\frac{1}{\w_1-\w_2}]$.
Givental's A-model canonical coordinates $(u^1,u^2)$ are characterized
by their large radius limits:
\begin{equation}\label{eqn:u-initial}
\lim_{q\to 0} (u^1 -t^0 -\w_1 t^1) =0,\quad
\lim_{q\to 0} (u^2-t^0-\w_2 t^1) =0.
\end{equation}

For $\alpha\in \{1,2\}$ and $i\in \{0,1\}$, define $\Psi_i^{\,\ \alpha}$ by
$$
\frac{du^\alpha}{\sqrt{\Delta^\alpha(q)}} =\sum_{i=0}^1 dt^i \Psi_i^{\,\ \alpha},
$$
and define the $\Psi$-matrix to be
$$
\Psi:= \left[ \begin{array}{cc} \Psi_0^{\,\ 1} & \Psi_0^{\,\ 2}\\
\Psi_1^{\,\ 1} & \Psi_1^{\,\ 2} \end{array}\right].
$$
Then
$$
\left[\displaystyle{\begin{array}{cc}
\frac{du^1}{ \sqrt{\Delta^1(q)} } & \frac{du^2}{ \sqrt{\Delta^2(q)} }
\end{array}}\right]
= \left[\begin{array}{cc} dt^0 & dt^1\end{array}\right]\Psi,
$$
$$
\Psi_0^{\,\ \alpha}= \frac{1}{\sqrt{\Delta^\alpha(q)}},\quad
\Psi_1^{\,\ \alpha}=\frac{\Delta^\alpha(q)+\w_1+\w_2}{2\sqrt{\Delta^\alpha(q)}}.
$$
Let
$$
\Psi^{-1}=\left[ \begin{array}{cc} (\Psi^{-1})_1^{\,\ 0} & (\Psi^{-1})_1^{\,\ 1}\\
(\Psi^{-1})_2^{\,\ 0} & (\Psi^{-1})_2^{\,\ 1} \end{array}\right]
$$
be the inverse matrix of $\Psi$, so that
$$
\sum_{i=0}^1 (\Psi^{-1})_\alpha^{\,\ i} \Psi_i^{\,\ \beta} =\delta_\alpha^{\,\ \beta}.
$$
Then
$$
(\Psi^{-1})_\alpha^{\,\ 0}=\frac{\Delta^\alpha(q)-\w_1-\w_2}{2\sqrt{\Delta^\alpha(q)}},\quad
(\Psi^{-1})_\alpha^{\,\ 1}=\frac{1}{\sqrt{\Delta^\alpha(q)}}.
$$
Let $Q=1$ i.e. $q=e^{t^1}$. We take the non-equivariant limit $\w_2=0$, $\w_1\to 0^+$:
$$
u^1= t^0+2\sqrt{q},\quad u^2=t^0-2\sqrt{q},
$$
\begin{eqnarray*}
\Psi &=& \frac{1}{\sqrt{2}}\left( \begin{array}{cc}
e^{-t^1/4} & -\sqrt{-1} e^{-t^1/4}\\
e^{t^1/4} & \sqrt{-1} e^{t^1/4}
\end{array} \right) ,\\
\Psi^{-1} &=& \frac{1}{\sqrt{2}}\left( \begin{array}{cc}
e^{t^1/4} &  e^{-t^1/4}\\
\sqrt{-1}e^{t^1/4} & -\sqrt{-1} e^{-t^1/4}
\end{array} \right).
\end{eqnarray*}
These non-equivariant limits agree with the results in \cite[Section 2]{SS}.

\subsection{The $\cS$-operator}\label{sec:A-S}
The $\cS$-operator is defined as follows.
For any cohomology classes $a,b\in H_T^*(\bP^1;\bC)$, 
$$
(a,\cS(b))=\llangle a,\frac{b}{z-\psi}\rrangle^{\bP^1,T}_{0,2}.
$$
The $T$-equivariant $J$-function is characterized by
$$
(J,a) = (1,\cS(a))
$$
for any $a\in H_T^*(\bP^1)$.

Let
$$
\chi^1=\w_1-\w_2,\quad \chi^2=\w_2-\w_1.
$$
We consider several different (flat) bases for $H_T^*(\bP^1;\bC)$:
\begin{enumerate}
\item The canonical basis: $\phi_1=\displaystyle{\frac{H-\w_2}{\w_1-\w_2}}$,
$\phi_2 = \displaystyle{\frac{H-\w_1}{\w_2-\w_1}}$.
\item The basis dual to the canonical basis with respect to the $T$-equivariant Poincare pairing:
$\phi^1 =\chi^1 \phi_1$, $\phi^2=\chi^2\phi_2$.
\item The normalized canonical basis $\hat{\phi}_1 = \sqrt{\chi^1}\phi_1$, $\hat{\phi}_2 =\sqrt{\chi^2}\phi_2$,
which is self-dual: $\hat{\phi}^1=\hat{\phi}_1$, $\hat{\phi}^2=\hat{\phi}_2$.
\item The natural basis: $T_0=1$, $T_1=H$.
\item The basis dual to the natural basis: $T^0=H-\w_1-\w_2$, $T^1=1$.
\end{enumerate}

For $\alpha,\beta\in \{1,2\}$, define
$$
S^\alpha_{\,\ \beta}(z) := (\phi^\alpha, \cS(\phi_\beta)).
$$
Then $S(z)=(S^\alpha_{\,\ \beta}(z))$ is the matrix\footnote{We use
the convention that the {\em left} superscript/subscript is the {\em row} number
and the {\em right} superscript/subscript is the {\em column} number.}  of the $\cS$-operator with respect to
the ordered basis $(\phi_1,\phi_2)$:
\begin{equation}\label{eqn:S}
\cS(\phi_\beta) =\sum_{\alpha=1}^2 \phi_\alpha S^\alpha_{\,\ \beta}(z).
\end{equation}

For $i\in \{0,1\}$ and $\alpha\in \{1,2\}$, define
$$
{S}_i^{\,\ \hat\alpha}(z) := (T_i, \cS(\hat{\phi}^\alpha)).
$$
Then $(S_i^{\,\ \hat \alpha})$ is the matrix of the $\cS$-operator
with respect to the ordered bases $(\hat{\phi}^1,\hat{\phi}^2)$ and
$(T^0,T^1)$:
\begin{equation}\label{eqn:barS}
\cS(\hat{\phi}^\alpha)=\sum_{i=0}^1 T^i S_i^{\,\ \hat{\alpha}}(z).
\end{equation}
We have
$$
z \frac{\partial J}{\partial t^i} = \sum_{\alpha=1}^2 {S}_i^{\,\ \hat \alpha}(z) \hat{\phi}_\alpha.
$$

By \cite{G96, LLY}, the equivariant $J$-function is
$$
J = e^{(t^0+t^1 H)/z} (1+ \sum_{d=1}^\infty \frac{q^d}{\prod_{m=1}^d(H-\w_1+mz)\prod_{m=1}^d(H-\w_2+mz)}).
$$

For $\alpha=1,2$, define
$$
J^\alpha := J|_{p_\alpha}= e^{(t^0+t^1\w^\alpha)/z} \sum_{d=0}^\infty \frac{q^d}{d!z^d}
\frac{1}{\prod_{m=1}^d(\chi^\alpha+mz)}.
$$
Then
$$
z\frac{\partial J}{\partial t^0} = J = \sum_{\alpha=1}^2 J^\alpha \phi_\alpha,\quad
z\frac{\partial J}{\partial t^1} = z\sum_{\alpha=1}^2 \frac{\partial J^\alpha}{\partial t^1}\phi_\alpha.
$$
So
$$
{S}_i^{\,\ \hat \alpha}(z)= \frac{z}{\sqrt{\chi^\alpha}}\cdot  \frac{\partial J^\alpha}{\partial t^i}.
$$
Following Givental, we define
$$
\tS_i^{\,\ \hat \alpha}(z) :=  S_i^{\,\ \hat \alpha}(z)\exp\big(-\sum_{n=1}^\infty \frac{B_{2n}}{2n(2n-1)}(\frac{z}{\chi^\alpha})^{2n-1}\big).
$$
Then
\begin{eqnarray*}
\tS_0^{\,\ \hat \alpha}(z) &=& \frac{1}{\sqrt{\chi^\alpha}} \exp\big(\frac{t^0+t^1\w_\alpha}{z}-\sum_{n=1}^\infty \frac{B_{2n}}{2n(2n-1)}(\frac{z}{\chi^\alpha})^{2n-1}\big)\cdot \Big( \sum_{d=0}^\infty \frac{q^d}{d!z^d} \frac{1}{\prod_{m=1}^d(\chi^\alpha+mz)}\Big)\\
\tS_1^{\,\ \hat \alpha}(z) &=& \frac{1}{\sqrt{\chi^\alpha}} \exp\big(\frac{t^0+t^1\w_\alpha}{z}-\sum_{n=1}^\infty \frac{B_{2n}}{2n(2n-1)}(\frac{z}{\chi^\alpha})^{2n-1}\big), \\
&&  \cdot \Big(\w_\alpha\sum_{d=0}^\infty \frac{q^d}{d!z^d}
\frac{1}{\prod_{m=1}^d(\chi^\alpha+mz)} + \sum_{d=1}^\infty \frac{q^d}{(d-1)!z^d}
\frac{1}{\prod_{m=1}^d(\chi^\alpha+mz)}\Big).
\end{eqnarray*}

\subsection{The A-model $R$-matrix}

By Givental \cite{G01}, the matrix $(\tS_i^{\,\ \hat \beta})(z)$ is of the form
$$
\tS_i^{\,\ \hbeta}(z)= \sum_{\alpha=1}^2 \Psi_i^{\ \alpha} R_\alpha^{\,\ \beta}(z) e^{u^\beta/z} =(\Psi R(z))_i^{\,\ \beta} e^{u^\beta/z},
$$
where $R(z)= (R_\alpha^{\,\ \beta}(z)) = I + \sum_{k=1}^\infty R_k z^k$ and is unitary, and
$$
\lim_{q\to 0} R_\alpha^{\,\ \beta}(z) =\delta_{\alpha\beta} \exp\big(-\sum_{n=1}^\infty \frac{B_{2n}}{2n(2n-1)}
(\frac{z}{\chi^\beta})^{2n-1} \big).
$$

\subsection{Gromov-Witten potentials}
Introducing formal variables
$$
\bu=\sum_{a\geq 0}u_a z^a,
$$
where
$$
u_a = \sum_{\alpha=1}^2 u_a^\alpha \phi_\alpha(q).
$$
Define
\begin{eqnarray*}
&& F^{\bP^1,T}_{g,n}(\bu,\bt) :=
\sum_{a_1,\ldots, a_n\in \bZ_{\geq 0}} \llangle \tau_{a_1}(u_{a_1})\cdots \tau_{a_n}(u_{a_n})\rrangle_{g,n}^{\bP^1,T} \\
&& \quad =  \sum_{a_1,\ldots, a_n\in \bZ_{\geq 0}}\sum_{m=0}^\infty \sum_{d=0}^\infty
\frac{Q^d}{n!m!}\int_{[\Mbar_{g,n+m}(\bP^1,d)]^\vir}
\prod_{j=1}^n \ev_j^*(u_{a_j}) \psi_j^{a_j}\prod_{i=1}^m\ev_{i+n}^*(\bt).
\end{eqnarray*}

We define the total descendent potential of $\bP^1$ to be
$$
D^{\bP^1,T}(\bu)=\exp (\sum_{n,g}\hbar^{g-1}F^{\bP^1,T}_{g,n}(\bu,0)).
$$
Consider the map $\pi:\Mbar_{g,n+m}(\bP^1,d)\to \Mbar_{g,n}$ which forgets the map to the target and the last $m$ marked points. Let $\bar\psi_i:=\pi^*(\psi_i)$ be the pull-backs of the classes $\psi_i, i=1,\cdots n$, from $\Mbar_{g,n}$. Then we can define
$$
\bar{F}^{\bP^1,T}_{g,n}(\bu,\bt) := \sum_{a_1,\ldots, a_n\in \bZ_{\geq 0}}\sum_{m=0}^\infty \sum_{d=0}^\infty
\frac{Q^d}{n!m!}\int_{[\Mbar_{g,n+m}(\bP^1,d)]^\vir}
\prod_{j=1}^n \ev_j^*(u_{a_j}) \bar\psi_j^{a_j}\prod_{i=1}^m\ev_{i+n}^*(\bt).
$$
Let the ancestor potential of $\bP^1$ be
$$A^{\bP^1,T}(\bu,\bt)=\exp (\sum_{n,g}\hbar^{g-1}\bar{F}^{\bP^1,T}_{g,n}(\bu,\bt)).$$

\subsection{Givental's formula for equivariant Gromov-Witten potential and
the A-model graph sum}
The quantization of the $\cS$-operator relates the ancestor potential and the descendent potential of $\bP^1$ via Givental's formula. Concretely, we have (see \cite{G01'})
$$D^{\bP^1,T}(\bu)=\exp(F^{\bP^1,T}_1)\hat{\cS}^{-1}A^{\bP^1,T}(\bu,\bt)$$
where $F^{\bP^1,T}_1$ denotes $\sum_{n}F^{\bP^1,T}_{1,n}(\bu,0)$ at $u_0=u, u_1=u_2=\cdots=0$ and $\hat{\cS}$ is the quantization \cite{G01'} of $\cS$. For our purpose, we need to describe a formula for a slightly different potential: $F^{\bP^1,T}_{g,n}(\bu,\bt)$---the descendent potential with arbitrary primary insertions.

Now we first describe a graph sum formula for the ancestor potential $A^{\bP^1,T}(\bu,\bt)$. Given a connected graph $\Ga$, we introduce the following notation.
\begin{enumerate}
\item $V(\Ga)$ is the set of vertices in $\Ga$.
\item $E(\Ga)$ is the set of edges in $\Ga$.
\item $H(\Ga)$ is the set of half edges in $\Gamma$.
\item $L^o(\Ga)$ is the set of ordinary leaves in $\Ga$.
\item $L^1(\Ga)$ is the set of dilaton leaves in $\Ga$.
\end{enumerate}

With the above notation, we introduce the following labels:
\begin{enumerate}
\item (genus) $g: V(\Ga)\to \bZ_{\geq 0}$.
\item (marking) $\beta: V(\Ga) \to \{1,2\}$. This induces
$\beta:L(\Ga)=L^o(\Ga)\cup L^1(\Ga)\to \{1,2\}$, as follows:
if $l\in L(\Ga)$ is a leaf attached to a vertex $v\in V(\Ga)$,
define $\beta(l)=\beta(v)$.
\item (height) $k: H(\Ga)\to \bZ_{\geq 0}$.
\end{enumerate}

Given an edge $e$, let $h_1(e),h_2(e)$ be the two half edges associated to $e$. The order of the two half edges does not affect the graph sum formula in this paper. Given a vertex $v\in V(\Ga)$, let $H(v)$ denote the set of half edges
emanating from $v$. The valency of the vertex $v$ is equal to the cardinality of the set $H(v)$: $\val(v)=|H(v)|$.
A labeled graph $\vGa=(\Ga,g,\beta,k)$ is {\em stable} if
$$
2g(v)-2 + \val(v) >0
$$
for all $v\in V(\Ga)$.

Let $\bGa(\bP^1)$ denote the set of all stable labeled graphs
$\vGa=(\Gamma,g,\beta,k)$. The genus of a stable labeled graph
$\vGa$ is defined to be
$$
g(\vGa):= \sum_{v\in V(\Ga)}g(v)  + |E(\Ga)|- |V(\Ga)|  +1
=\sum_{v\in V(\Ga)} (g(v)-1) + (\sum_{e\in E(\Gamma)} 1) +1.
$$
Define
$$
\bGa_{g,n}(\bP^1)=\{ \vGa=(\Gamma,g,\beta,k)\in \bGa(\bP^1): g(\vGa)=g, |L^o(\Ga)|=n\}.
$$
Given $\alpha\in \{1,2\}$, define
$$
\bu^\alpha(z) = \sum_{a\geq 0} u^\alpha_a z^a.
$$

We assign weights to leaves, edges, and vertices of a labeled graph $\vGa\in \bGa(\bP^1)$ as follows.
\begin{enumerate}
\item {\em Ordinary leaves.} To each ordinary leaf $l \in L^o(\Ga)$ with  $\beta(l)= \beta\in \{1,2\}$
and  $k(l)= k\in \bZ_{\geq 0}$, we assign:
$$
(\cL^{\bu})^\beta_k(l) = [z^k] (\sum_{\alpha=1,2}\frac{\bu^\alpha(z)}{\sqrt{\Delta^\alpha(q)}}
R_\alpha^{\,\ \beta}(-z) ).
$$
\item {\em Dilaton leaves.} To each dilaton leaf $l \in L^1(\Ga)$ with $\beta(l)=\beta \in \{1,2\}$
and $2\leq k(l)=k \in \bZ_{\geq 0}$, we assign
$$
(\cL^1)^\beta_k(l) = [z^{k-1}](-\sum_{\alpha=1,2}\frac{1}{\sqrt{\Delta^\alpha(q)}} R_\alpha^{\,\ \beta}(-z)).
$$

\item {\em Edges.} To an edge connecting a vertex marked by $\alpha\in \{1,2\}$ to a vertex
marked by $\beta\in \{1,2\}$ and with heights $k$ and $l$ at the corresponding half-edges, we assign
$$
\cE^{\alpha,\beta}_{k,l}(e) = [z^k w^l]
\Bigl(\frac{1}{z+w} (\delta_{\alpha,\beta}-\sum_{\gamma=1,2}
R_\gamma^{\,\ \alpha}(-z) R_\gamma^{\,\ \beta}(-w)\Bigr).
$$
\item {\em Vertices.} To a vertex $v$ with genus $g(v)=g\in \bZ_{\geq 0}$ and with
marking $\beta(v)=\beta$, with $n$ ordinary
leaves and half-edges attached to it with heights $k_1, ..., k_n \in \bZ_{\geq 0}$ and $m$ more
dilaton leaves with heights $k_{n+1}, \ldots, k_{n+m}\in \bZ_{\geq 0}$, we assign
$$
\big(\sqrt{\Delta^\beta(q)}\big)^{2g-2+n+m} \int_{\Mbar_{g,n+m}}\psi_1^{k_1} \cdots \psi_{n+m}^{k_{n+m}}.
$$
\end{enumerate}

We define the weight of a labeled graph $\vGa\in \bGa(\bP^1)$ to be
\begin{eqnarray*}
w(\vGa) &=& \prod_{v\in V(\Ga)} \big(\sqrt{\Delta^{\beta(v)}(q)}\big)^{2g(v)-2+\val(v)} \langle \prod_{h\in H(v)} \tau_{k(h)}\rangle_{g(v)}
\prod_{e\in E(\Ga)} \cE^{\beta(v_1(e)),\beta(v_2(e))}_{k(h_1(e)),k(h_2(e))}(e)\\
&& \cdot \prod_{l\in L^o(\Ga)}(\cL^{\bu})^{\beta(l)}_{k(l)}(l)
\prod_{l\in L^1(\Ga)}(\cL^1)^{\beta(l)}_{k(l)}(l).
\end{eqnarray*}
Then
$$
\log(A^{\bP^1,T}(\bu,\bt)) =
\sum_{\vGa\in \bGa(\bP^1)}  \frac{ \hbar^{g(\vGa)-1} w(\vGa)}{|\Aut(\vGa)|}
= \sum_{g\geq 0}\hbar^{g-1} \sum_{n\geq 0}\sum_{\vGa\in \bGa_{g,n}(\bP^1)}\frac{w(\vGa)}{|\Aut(\vGa)|}.
$$

Now we describe a graph sum formula for $F^{\bP^1,T}_{g,n}(\bu,\bt)$---the descendant potential with arbitrary primary insertions. For $\alpha=1,2$, let
$$
\hat{\phi}_\alpha(q):= \sqrt{\Delta^\alpha(q)}\phi_\alpha(q).
$$
Then $\hat{\phi}_1(q)$, $\hat{\phi}_2(q)$ is the normalized canonical basis
of $QH_T^*(\bP^1;\bC)$, the $T$-equivariant quantum cohomology of $\bP^1$. Define
$$
{S}^{\hualpha}_{\, \ \hubeta}(z) := (\hat{\phi}_\alpha(q), \cS(\hat{\phi}_\beta(q))).
$$
Then $({S}^{\underline{\hat \alpha}}_{\, \ \underline{\hat{\beta}}}(z))$ is the matrix of the $\cS$-operator with
respect to the ordered basis $(\hat{\phi}_1(q),\hat{\phi}_2(q))$:
\begin{equation}\label{eqn:mathringS}
\cS(\hat{\phi}_\beta(q))=\sum_{\alpha=1}^2 \hat{\phi}_\alpha(q) {S}^\hualpha_{\,\ \hubeta}(z).
\end{equation}
We define a new weight of the ordinary leaves:
\begin{enumerate}
\item[(1)'] {\em Ordinary leaves.}
To each ordinary leaf $l \in L^o(\Ga)$ with  $\beta(l)= \beta\in \{1,2\}$
and  $k(l)= k\in \bZ_{\geq 0}$, we assign:
$$
(\mathring{\cL}^{\bu})^\beta_k(l) = [z^k] (\sum_{\alpha,\gamma=1,2} \frac{\bu^\alpha(z)}{\sqrt{\Delta^\alpha(q)}}
{S}^{\hugamma}_{\,\ \hualpha}(z) R(-z)_\gamma^{\,\ \beta} ).
$$
\end{enumerate}
We define a new weight of a labeled graph $\vGa\in \bGa(\bP^1)$ to be
\begin{eqnarray*}
\mathring{w}(\vGa) &=& \prod_{v\in V(\Ga)} (\sqrt{\Delta^{\beta(v)}(q)})^{2g(v)-2+\val(v)} \langle \prod_{h\in H(v)} \tau_{k(h)}\rangle_{g(v)}
\prod_{e\in E(\Ga)} \cE^{\beta(v_1(e)),\beta(v_2(e))}_{k(h_1(e)),k(h_2(e))}(e)\\
&& \cdot \prod_{l\in L^o(\Ga)}(\mathring{\cL}^{\bu})^{\beta(l)}_{k(l)}(l)
\prod_{l\in L^1(\Ga)}(\cL^1)^{\beta(l)}_{k(l)}(l).
\end{eqnarray*}
Then
$$
\sum_{g\geq 0}\hbar^{g-1} \sum_{n\geq 0}F^{\bP^1,T}_{g,n}(\bu,\bt) =
\sum_{\vGa\in \bGa(\bP^1)}  \frac{ \hbar^{g(\vGa)-1} \mathring{w}(\vGa)}{|\Aut(\vGa)|}
= \sum_{g\geq 0}\hbar^{g-1} \sum_{n\geq 0}\sum_{\vGa\in \bGa_{g,n}(\bP^1)}\frac{\mathring{w}(\vGa)}{|\Aut(\vGa)|}.
$$

We can slightly generalize this graph sum formula to the case where we have $n$ \emph{ordered} variables $\bu_1,\cdots,\bu_n$ and $n$ \emph{ordered} ordinary leaves. Let
$$
\bu_j=\sum_{a\geq 0}(u_j)_a z^a
$$
and let
$$
F^{\bP^1,T}_{g,n}(\bu_1,\cdots,\bu_n,\bt) := \sum_{a_1,\ldots, a_n\in \bZ_{\geq 0}}\sum_{m=0}^\infty \sum_{d=0}^\infty
\frac{1}{m!}\int_{[\Mbar_{g,n+m}(\bP^1,d)]^\vir}
\prod_{j=1}^n \ev_j^*((u_j)_{a_j}) \psi_j^{a_j}\prod_{i=1}^m\ev_{i+n}^*(\bt).
$$
Define the set of graphs $\tilde{\bGa}_{g,n}(\bP^1)$ as the definition of $\bGa_{g,n}(\bP^1)$ except that the $n$ ordinary leaves are \emph{ordered}. Let $\{l_1,\cdots,l_n\}$ be the ordinary leaves in $\Gamma\in \tilde{\bGa}_{g,n}(\bP^1)$ and for $j=1,\cdots,n$ let
$$
(\mathring{\cL}^{\bu_j})^\beta_k(l_j) = [z^k] (\sum_{\alpha,\gamma=1,2} \frac{\bu^\alpha_j(z)}{\sqrt{\Delta^\alpha(q)}}
{S}^{\hugamma}_{\,\ \hualpha}(z) R(-z)_\gamma^{\,\ \beta} ).
$$
Define the weight
\begin{eqnarray*}
\mathring{w}(\vGa) &=& \prod_{v\in V(\Ga)} (\sqrt{\Delta^{\beta(v)}(q)})^{2g(v)-2+\val(v)} \langle \prod_{h\in H(v)} \tau_{k(h)}\rangle_{g(v)}
\prod_{e\in E(\Ga)} \cE^{\beta(v_1(e)),\beta(v_2(e))}_{k(h_1(e)),k(h_2(e))}(e)\\
&& \cdot \prod_{j=1}^{n}(\mathring{\cL}^{\bu_j})^{\beta(l_j)}_{k(l_j)}(l_j)
\prod_{l\in L^1(\Ga)}(\cL^1)^{\beta(l)}_{k(l)}(l).
\end{eqnarray*}
Then
$$
\sum_{g\geq 0}\hbar^{g-1} \sum_{n\geq 0}F^{\bP^1,T}_{g,n}(\bu_1,\cdots,\bu_n,\bt) =
\sum_{\vGa\in \tilde{\bGa}(\bP^1)}  \frac{ \hbar^{g(\vGa)-1} \mathring{w}(\vGa)}{|\Aut(\vGa)|}
= \sum_{g\geq 0}\hbar^{g-1} \sum_{n\geq 0}\sum_{\vGa\in \tilde{\bGa}_{g,n}(\bP^1)}\frac{\mathring{w}(\vGa)}{|\Aut(\vGa)|}.
$$

\section{B-model}

\subsection{The equivariant superpotential and the Frobenius structure of the Jacobian ring}
Let $Y$ be coordinates on $\bC^*$.
The $T$-equivariant superpotential $W_t^\w: \bC^*\to \bC$ is given by
\begin{eqnarray*}
W_t^\w(Y)&=& Y+ t_0  +\frac{q}{Y} + \w_1 \log Y + \w_2\log\frac{q}{Y},
\end{eqnarray*}
where $q=Qe^{t_1}$ and $Y = e^y$. In this section, we assume $\w_1-\w_2$ is a positive real number.
The Jacobian ring of $W^\w_t$ is
$$
\Jac(W_t^{\w}) \cong C[Y,Y^{-1},q,\w]/\langle \frac{\partial W_t^\w}{\partial y} \rangle = \bC[Y, Y^{-1},q,\w]/\langle Y-\frac{q}{Y} +\w_1-\w_2\rangle
$$
Let
$$
B := q\frac{\partial W_t^{\w}}{\partial q}  = \frac{q}{Y}+ \w_2.
$$
%
The Jacobian ring is isomorphic to $QH^*_T(\bP^1;\bC)$ if one identifies $B$ with $H$
$$
\Jac(W_t^{\w}) \cong \bC[B,q,\w]/\langle(B-\w_1)(B-\w_2)-q\rangle.
$$

The critical points of $W_t^\w$ are $P_1$, $P_2$, where
$$
P_\alpha = \frac{\w_2-\w_1+\Delta^\alpha(q)}{2},\quad \alpha=1,2.
$$
Endow a metric on $\Jac(W_q^{\w})$ by the residue pairing
$$
(f,g)=\sum_{\alpha=1}^2 \Res_{Y=P_\alpha} \frac{f(Y)g(Y)} {\frac{\partial W_t^\w}{\partial y}} \frac{dY}{Y}.
$$
By direct calculation, we have
$$(B,B)=\w_1+\w_2,\quad (B,\one)=(\one,B)=1,\quad (\one,\one)=0.$$
We denote $b_0=\one$, $b_1=B$ and $b^i$ by $(b^i, b_j)=\delta^i_j$.
These calculations show the following well-known fact.
\begin{proposition}
There is an isomorphism of Frobenius manifold
$$
QH^*_T(\bP^1;\bC)\otimes_{\bC[\w]}\bC[\w,\frac{1}{\w_1-\w_2}] \cong \Jac(W^\w_t)\otimes_{\bC[\w]}\bC[\w,\frac{1}{\w_1-\w_2}].
$$
\end{proposition}
We denote $\Jac(W^\w_t) \otimes_{\bC[\w]}\bC[\w,\frac{1}{\w_1-\w_2}]$ by $H_B$. The Dubrovin-Givental connection is denoted by $\nabla^B_{v}=z\partial_v + v\bullet$ on $\cH_B:=H_B \lp z\rp$.

\subsection{The B-model canonical coordinates} \label{sec:B-canonical}
The isomorphism of Frobenius structures automatically ensures their canonical coordinates are the same up to a permutation and constants. We fix the B-model canonical coordinates in this subsection by  the critical values of the superpotential $W^\w_t$, and find the constant difference to the A-model coordinates that we set up in earlier sections.

Let $C_t^\w =\{(x,y) \in \bC^2: x = W_t^\w(e^y) \}$ be the graph of the equivariant superpotential. It is a covering of $\bC^*$ given by $y \mapsto e^y$. Let $\bar\Si \cong \bP^1$ be the compactification of $\bC^*$ with $Y\in \bC^*\subset \bP^1$ as its coordinate. At each branch point $Y=P_\alpha$, $x$ and $y$ have the following expansion
\begin{align*}
x&=\check{u}^\alpha-\zeta_\alpha^2,\\
y&=\check{v}^\alpha-\sum_{k=1}^\infty h^\alpha_k(q) \zeta_\alpha^k,
\end{align*}
where $h^\alpha_1(q)=\sqrt{\frac{2}{\Delta^\alpha(q)}}$. Note that
we define $\zeta_\alpha$ by $\zeta_\alpha^2 = \check{u}^\alpha-x$, which differs
from the definition of $\zeta$ in \cite{E11, EO12} by a factor of $\sqrt{-1}$.

The critical values are
$$
\check u^\alpha=t^0+\w_\alpha t^1 +\Delta^\alpha(q)-\chi^\alpha\log\frac{\chi^\alpha+\Delta^\alpha(q)}{2}.
$$
Since
\begin{eqnarray*}
\frac{\partial \check{u}^\alpha}{\partial t^0} &=&  1,\\
\frac{\partial\check{u}^\alpha}{\partial t^1} &=& \frac{q}{P_\alpha} + \w_2 = \frac{\w_1+\w_2+\Delta^\alpha(q)}{2},
\end{eqnarray*}
we have
\begin{equation}\label{eqn:AB-du}
d\check{u}^\alpha = du^\alpha, \quad \alpha=1,2.
\end{equation}
Recall that $\lim_{q\to 0} \Delta^1(q) = \w_1-\w_2$, so in the large radius limit $q\to 0$, we have
\begin{eqnarray}
\lim_{q\to 0} (\check u^\alpha-t^0-\w_\alpha t^1)=\chi^\alpha -\chi^\alpha \log \chi^\alpha.
\label{eqn:u}
\end{eqnarray}
From \eqref{eqn:AB-du},\eqref{eqn:u},  and \eqref{eqn:u-initial}, we conclude that
$$
\check{u}^\alpha = u^\alpha + a_\alpha,\quad \alpha=1,2,
$$
where
\begin{eqnarray*}
a_\alpha=\chi^\alpha -\chi^\alpha \log \chi^\alpha.
\end{eqnarray*}

\subsection{The Liouville form and Bergman kernel}
On $C_t^\w$, let
$$
\lambda = x dy
$$
be the Liouville form on $\bC^2 = T^*\bC$. Then $d\lambda = dx\wedge dy$.
Let
$$
\Phi := \lambda|_{C_t^\w} = W^\w_t(e^y)dy = (e^y + t_0 + q e^{-y} + (\w_1-\w_2)y + \w_2 \log q) dy.
$$
Then $\Phi$ is a holomorphic 1-form on $\bC$. Recall that $q=Qe^{t^1}$ and $Y=e^y$.
Define
\begin{eqnarray*}
\Phi_0 &:=& \frac{\partial \Phi}{\partial t^0} = \frac{dY}{Y}, \\
\Phi_1 &:=& \frac{\partial \Phi}{\partial t^1} = (\frac{q}{Y} +\w_2)\frac{dY}{Y}.
\end{eqnarray*}
Then $\Phi_0,\Phi_1$ descends to holomorphic 1-forms on $\bC^*$ which extends to
meromorphic 1-forms on $\bP^1$. We have
\begin{itemize}
\item $\Phi_0$  has simple poles at $Y=0$ and $Y=\infty$, and
$$
\Res_{Y\to 0}\Phi_0=1,\quad \Res_{Y\to \infty}\Phi_1=-1.
$$
\item $\Phi_1-\w_2\Phi_0 = -q d(Y^{-1})$ is an exact 1-form.
\end{itemize}

Let $B(p_1,p_2)$ be the fundamental normalized differential of the second kind on
$\bar{\Si}$ (see e.g. \cite{Fay}). It is also called the Bergman kernel in \cite{EO07,EO12}. In this simple case $\bar\Si\cong \bP^1$, we have
$$
B(Y_1,Y_2)=\frac{dY_1\otimes dY_2}{(Y_1-Y_2)^2}.
$$

\subsection{Differentials of the second kind}

Following \cite{E11, EO12}, given $\alpha=1,2$ and $d\in \bZ_{\geq 0}$, define
$$
d\xi_{\alpha,d}(p):= (2d-1)!! 2^{-d}\Res_{p'\to P_\alpha}
B(p,p')(\sqrt{-1}\zeta_\alpha)^{-2d-1}.
$$
Then $d\xi_{\alpha,d}$ satisfies the following properties.
\begin{enumerate}
\item $d\xi_{\alpha,d}$ is a meromorphic 1-form on $\bP^1$ with
a single pole of order $2d+2$ at $P_\alpha$.
\item In local coordinate $\zeta_\alpha$ near $P_\alpha$,
$$
d\xi_{\alpha,d} = \Big( \frac{-(2d+1)!!}{2^d \sqrt{-1}^{2d+1}\zeta_\alpha^{2d+2}}
+ f(\zeta_\alpha)\Big) d\zeta_\alpha,
$$
where $f(\zeta_\alpha)$ is analytic around $P_\alpha$.
The residue of $d\xi_{\alpha,d}$ at $P_\alpha$ is zero,
so $d\xi_{\alpha,d}$ is a differential of the second kind.
\end{enumerate}
The meromorphic 1-form $d\xi_{\alpha,d}$ is characterized by the above
properties; $d\xi_{\alpha,d}$ can be viewed as a section in
$H^0(\bP^1, \omega_{\bP^1}((2d+2) P_\alpha) )$.
In particular, $d\xi_{\alpha,0}$ is
$$
d\xi_{\alpha,0} = \frac{1}{\sqrt{-1}}\sqrt{\frac{2}{\Delta^\alpha(q)}} d(\frac{P_\alpha}{Y-P_
\alpha}).
$$
Then we have
\begin{eqnarray*}
d(\frac{\Phi_0}{dW}) &=& d (\frac{Y}{(Y-P_1)(Y-P_2)}) = \frac{1}{P_1-P_2} d(\frac{P_1}{Y-P_1} - \frac{P_2}{Y-P_2}) \\
&=& \frac{1}{\sqrt{-1}}\frac{1}{\sqrt{2\Delta^1(q)}} d\xi_{1,0} +\frac{1}{\sqrt{-1}}\frac{1}{\sqrt{2\Delta^2(q)}} d\xi_{2,0} \\
&=& \frac{1}{\sqrt{-2}} \sum_{\alpha=1}^2 \Psi_0^{\,\ \alpha}d\xi_{\alpha,0},\\
d(\frac{\Phi_1}{dW}) &=& d (\frac{q+\w_2 Y}{(Y-P_1)(Y-P_2)})  \\
&=&  \frac{1}{P_1-P_2}d(\frac{q+ P_1\w_2}{Y-P_1} - \frac{q+ P_2\w_2}{Y-P_2}) \\
&=& \frac{1}{\sqrt{-1}}\frac{1}{\Delta^1(q)}\Big( \sqrt{\frac{\Delta^1(q)}{2}} (\frac{q}{P_1} +\w_2)d\xi_{1,0} -
\sqrt{\frac{\Delta^2(q)}{2}}(\frac{q}{P_2}+\w_2)d\xi_{2,0}\Big)\\
&=& \frac{1}{2\sqrt{-2}}\Big((\sqrt{\Delta^1(q)} + \frac{\w_1+\w_2}{\sqrt{\Delta^1(q)}}) d\xi_{1,0}
+ \frac{1}{2\sqrt{2}}(\sqrt{\Delta^2(q)}+ \frac{\w_1+\w_2}{\sqrt{\Delta^2(q)}} ) d\xi_{2,0}\big)\\
&=& \frac{1}{\sqrt{-2}}\sum_{\alpha=1}^2\Psi_1^{\,\ \alpha} d\xi_{\alpha,0}.
\end{eqnarray*}
We have
\begin{equation}
\left(\begin{array}{c}
d(\frac{\Phi_0}{dW}) \\ d(\frac{\Phi_1}{dW}) \end{array}\right)
=\frac{1}{\sqrt{-2}}\Psi\left(\begin{array}{c} d\xi_{1,0}\\d\xi_{2,0}\end{array}\right),\quad
\sqrt{-2} \Psi^{-1} \left(\begin{array}{c}
d(\frac{\Phi_0}{dW}) \\ d(\frac{\Phi_1}{dW}) \end{array}\right)
=\left(\begin{array}{c} d\xi_{1,0}\\d\xi_{2,0}\end{array}\right).
\label{eqn:B-model-Psi}
\end{equation}

\subsection{Oscillating integrals and the B-model $R$-matrix}
For $\alpha,\beta\in \{1,2\}$, $i\in \{0,1\}$ and $z>0$, define
$$
\check{S}_i^{\,\ \balpha}(z)   :=  \int_{y\in \gamma_\alpha} e^\frac{W_q^\w(Y)}{z} \Phi_i
= -z \int_{y\in \gamma_\alpha} e^\frac{W_q^\w(Y)}{z}d(\frac{\Phi_i}{dW}),
$$
where $\gamma_\alpha$ is the Lefschetz thimble going through $P_\alpha$, such that $W^\w_q(Y)\to -\infty$ near its ends.
It is straightforward to check that $\sum_{i=0}^1 b^i \check{S}_i^{\ \balpha}$ is a solution to the quantum differential equation $\nabla^B f=0$ for $\alpha=1,2$. We quote the following theorem
\begin{theorem}[\cite{D93,G97,G01}]
\label{thm:factorize-S}
Near a semi-simple point on a Frobenius manifold of dimension $n$, there is a fundamental solution $S$ to the quantum differential equation satisfying the following properties
\begin{enumerate}
\item $S$ has the following form

$$
S=\Psi R(z) e^{U/z},
$$
where $R(z)$ is matrix of formal power series in $z$, and $U=\mathrm{diag}(u^1,\dots,u^n)$ is a matrix formed by canonical coordinates.
\item If $S$ is unitary under the pairing of the Frobenius structure, then $R(z)$ is unique up to a right multiplication of $e^{\sum_{i=1}^\infty {A_{2i-1} z^{2i-1}}}$ where $A_k$ are constant diagonal matrices.
\end{enumerate}
\end{theorem}

\begin{remark}\label{A-model S} For equivariant Gromov-Witten theory of $\bP^1$,
the fundamental solution $S$ in Theorem \ref{thm:factorize-S} is viewed as a matrix with entries in
$\bC[\w,\frac{1}{\w_1-\w_2}]((z))[[q,t^0,t^1]]$. We choose the canonical coordinates $\{u^\alpha(t)\}$ such that
there is no constant term by Equation \eqref{eqn:u-initial}.
Then if we fix the powers of $q,t^0$ and $t^1$, only finitely many terms in the expansion of $e^{U/z}$ contribute. So the multiplication $\Psi R(z)e^{U/z}$ is well defined and the result matrix indeed has entries in
$\bC[\w,\frac{1}{\w_1-\w_2}]((z))[[q,t^0,t^1]]$.
\end{remark}

\begin{remark}
For a general abstract semi-simple Frobenius manifold defined over a ring $A$, the expression $S=\Psi R(z) e^{U/z}$ in Theorem \ref{thm:factorize-S} can be understood in the following way. We consider the free module $M=\langle e^{u^1/z}\rangle\oplus\cdots\oplus \langle e^{u^n/z}\rangle$ over the ring $A((z))[[t^1,\cdots,t^n]]$ where $t^1,\cdots,t^n$ are the flat coordinates of the Frobenius manifold. We formally define the differential $de^{u^i/z}=e^{u^i/z}\frac{du^i}{z}$ and we extend the differential to $M$ by the product rule. Then we have a map $d:M\to Mdt^1\oplus\cdots\oplus Mdt^n$. We consider the fundamental solution $S=\Psi R(z) e^{U/z}$ as a matrix with entries in $M$. The meaning that $S$ satisfies the quantum differential equation is understood by the above formal differential.

In our case, the multiplication in the A-model fundamental solution $S=\Psi R(z) e^{U/z}$ is formal in $z$ as in Remark \ref{A-model S}. On B-model side, we use the stationary phase expansion to obtain a product of the form $\Psi R(z) e^{U/z}$. The multiplications $\Psi R(z) e^{U/z}$ on both A-model and B-model can be viewed as matrices with entries in $M$. And their differentials are obvious the same with the formal differential above.
\end{remark}

We repeat the argument in Givental \cite{G01'} and state it as the following fact.
\begin{proposition}
The fundamental solution matrix $\{\frac{\check{S}_i^{\ \balpha}}{\sqrt{-2\pi z}}\}$ has the following asymptotic expansion where $\check R(z)$ is a formal power series in $z$
$$
\frac{\check S_i^{\ \balpha}(z)}{\sqrt{-2\pi z}}\sim \sum_{\gamma=1}^2 \Psi_i^{\ \gamma}{\check R}_{\gamma}^{\ \alpha}(z) e^{\frac{\check u^\alpha}{z}}.
$$
\end{proposition}
\begin{proof}
By the stationary phase expansion,
$$
\check S_i^{\ \alpha}(z) \sim \sqrt{2\pi z} e^\frac{\check{u}^\alpha}{z} (1+a^{\ \alpha}_{i,1} z+ a^{\ \alpha}_{i,2} z^2 +\dots),
$$
it follows that $\{\check S_i^{\ \alpha}\}$ can be asymptotically expanded in the desired form (notice that $\Psi$ is a matrix in $z$-degree $0$). In particular, by \eqref{eqn:B-model-Psi}
$$
\check R_\beta^{\ \alpha}(z)\sim\frac{\sqrt{z}e^{-\frac{\check u^\alpha}{z}}}{2\sqrt{\pi }}\int_{\gamma_\alpha} e^{\frac{W^\w_t}{z}} d\xi_{\beta,0}.
$$

Following Eynard \cite{E11}, define Laplace transform of the Bergman kernel
\begin{equation}
\check{B}^{\alpha,\beta}(u,v,q) :=\frac{uv}{u+v} \delta_{\alpha,\beta}
+ \frac{\sqrt{uv}}{2\pi} e^{u\check u^\alpha + v \check u^\beta}\int_{p_1\in \gamma_\alpha}\int_{p_2\in \gamma_\beta}
B(p_1,p_2) e^{-ux(p_1) -v x(p_2)},
\end{equation}
where $\alpha,\beta\in \{1,2\}$. By \cite[Equation (B.9)]{E11},
\begin{equation}\label{eqn:B-ff}
                                                                                                           \check{B}^{\alpha,\beta}(u,v,q) = \frac{uv}{u+v}(\delta_{\alpha,\beta}
-\sum_{\gamma=1}^2 \check R^{\ \alpha}_\gamma(-\frac{1}{u})\check R^{\ \beta}_\gamma(-\frac{1}{v})).
\end{equation}
Setting $u=-v$, we conclude that $(\check R^*(\frac{1}{u}) \check R(-\frac{1}{u}))^{\alpha\beta}=\{\sum_{\gamma=1}^2 \check R^{\ \alpha}_\gamma(\frac{1}{u})\check R^{\ \beta}_\gamma(-\frac{1}{u})\}=\delta^{\alpha\beta}$. This shows $\check R$ is unitary.
\end{proof}

Following Iritani \cite{Ir09} (with slight modification), we introduce the following definition.
\begin{definition}[equivariant K-theoretic framing]\label{K-framing}
We define $\tch_z: K_T(\bP^1)\to H^*_T(\bP^1;\mathbb{Q})[[\frac{\w_1-\w_2}{z}]]$ by the following two properties
which uniquely characterize it.
\begin{itemize}
\item[(a)] $\tch_z$ is a homomorphism of additive groups:
$$
\tch_z(\cE_1\oplus \cE_2)=\tch_z(\cE_1) + \tch_z(\cE_2).
$$
\item[(b)] If $\cL$ is a $T$-equivariant line bundle on $\bP^1$ then
$$
\tch_z(\cL) = \exp\big(-\frac{2\pi\sqrt{-1}(c_1)_T(\cL)}{z}\big).
$$
\end{itemize}
For any $\cE\in K_T(\bP^1)$, we define the {\em K-theoretic framing} of $\cE$ by
$$
\kappa(\cE) :=(-z)^{1- \frac{(c_1)_T(T\bP^1)}{z}}\Gamma(1-\frac{(c_1)_T(T\bP^1)}{z}) \tch_z(\cE)
$$
where $(c_1)_T(T\bP^1)= 2H-\w_1-\w_2$.
\end{definition}
By localization, property (b) in the above definition is characterized by
$$
\iota_{p_\alpha}^*\kappa(\cO_{\bP^1}(l_1p_1+l_2 p_2))
=(-z)^{1-\frac{\chi^\alpha}{z}}\Gamma(1-\frac{\chi^\alpha}{z})e^{\frac{-2l_\alpha \pi\sqrt{-1}\chi^\alpha}{z}},\quad \alpha=1,2,
$$
where $\iota_{p_\alpha}:p_\alpha\to \bP^1$ is the inclusion map.

The following definition is motivated by \cite{F08, FLTZ12}.
\begin{definition}[equivariant SYZ T-dual]\label{SYZ-dual}
Let $\cL=\cO_{\bP^1}(l_1 p_1 + l_2 p_2)$ be an equivariant ample line bundle on $\bP^1$, where $l_1, l_2$
are integers such that $l_1+l_2>0$. We define the equivariant SYZ T-dual $\SYZ(\cL)$ of $\cL$ to be the
oriented graph in Figure 1 below.
\begin{figure}[h]
\begin{center}
\psfrag{infty}{\footnotesize $+\infty + (2 l_2-1)\pi i $}
\psfrag{-infty}{\footnotesize $-\infty + (-2 l_1-1)\pi i$}
\psfrag{L1}{\footnotesize $(-2l_1-1)\pi i $}
\psfrag{L2}{\footnotesize $(2l_2-1)\pi  i$}
\includegraphics[scale=0.5]{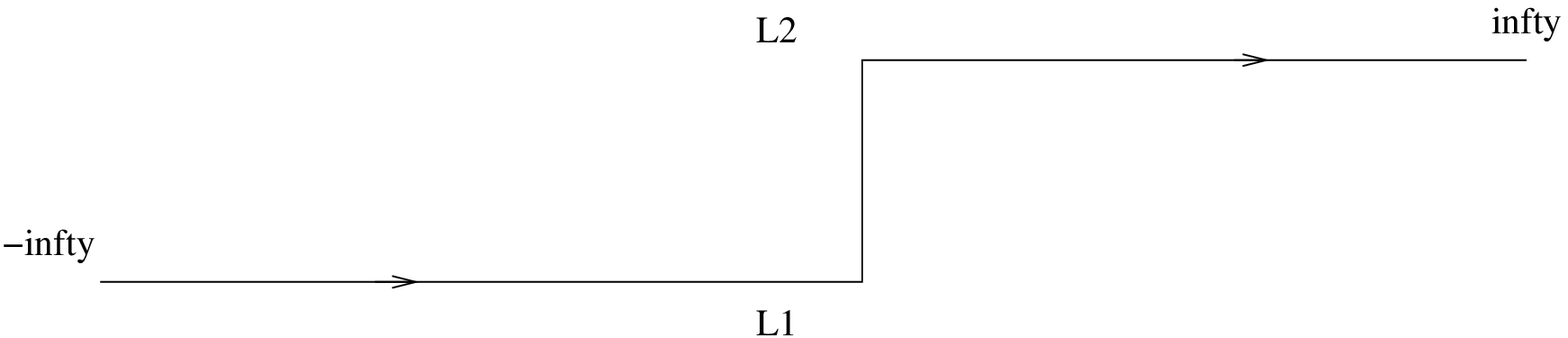}
\caption{$\SYZ(\cO_{\bP^1}(l_1 p_1+l_2 p_2))$ in $\bC$}
\end{center}
\end{figure}
We extend the definition additively to the equivariant K-theory group $K_T(\bP^1)$.
\end{definition}


\begin{figure}[h]
\begin{center}
\psfrag{pi}{\footnotesize $\pi i$}
\psfrag{-pi}{\footnotesize $-\pi i$}
\psfrag{infty}{\footnotesize $+\infty+\pi i$}
\psfrag{-infty}{\footnotesize $-\infty-\pi i$}
\psfrag{-inf}{\footnotesize $-\infty$}
\psfrag{exp}{\small $\exp$}
\psfrag{0}{\footnotesize $0$}
\psfrag{1}{\footnotesize $1$}
\psfrag{SYZT}{\small $\SYZ(\cO_{\bP^1}(p_2))$\textup{ in }$\bC$}
\psfrag{SYZ}{\small $\SYZ(\cO_{\bP^1}(1))$\textup{ in }$\bC^*$}
\includegraphics[scale=0.5]{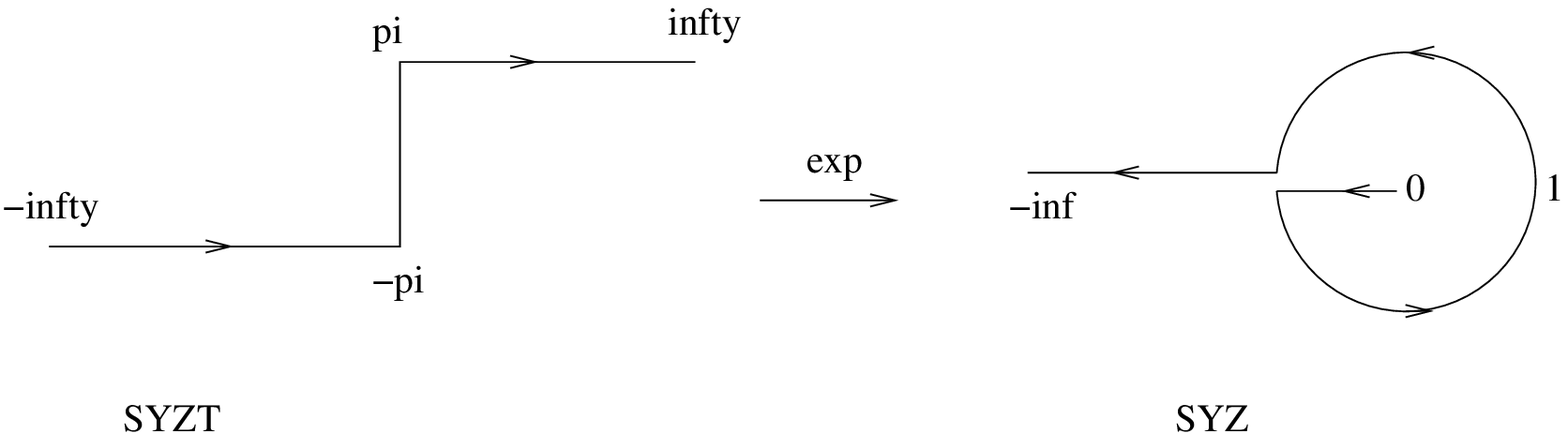}
\caption{The equivariant SYZ T-dual of $\cO_{\bP^1}(p_2)$ in $\bC$ and
the (non-equivariant) SYZ T-dual of $\cO_{\bP^1}(1)$ in $\bC^*$.}
\end{center}
\end{figure}

The following theorem gives a precise correspondence between
the B-model oscillatory integrals and the A-model 1-point descendant invariants.
\begin{theorem} \label{central-charge}
Suppose that $z, q, \w_1-\w_2\in (0,\infty)$. Then for any $\cL \in K_T(\bP^1)$,
\begin{equation}\label{eqn:central-charge}
\int_{y\in \SYZ(\cL)} e^{ {\frac{W^\w_t}{z}}} dy=  \llangle 1, \frac{\kappa(\cL)}{z-\psi} \rrangle^{\bP^1,T}_{0,2}.
\end{equation}
\begin{equation}\label{eqn:laplace-lambda}
\int_{y\in \SYZ(\cL)} e^{\frac{W^\w_t}{z}} ydx  =  - \llangle  \frac{\kappa(\cL)}{z-\psi} \rrangle^{\bP^1,T}_{0,1}.
\end{equation}
Here $dx=d(W^\w_t(y))$.
\end{theorem}

\begin{proof}
The left hand side of \eqref{eqn:central-charge} is
$$
\int_{y\in \SYZ(\cL)}e^{\frac{W_t^\w}{z}}dy = -\frac{1}{z} \int_{y\in \SYZ(\cL)} e^{\frac{W^\w_t}{z}} y d(W^\w_t).
$$
By the string equation, the right hand side of \eqref{eqn:central-charge} is
$$
\llangle 1, \frac{\kappa(\cL)}{z-\psi} \rrangle^{\bP^1,T}_{0,2}=\llangle \frac{\kappa(\cL)}{z(z-\psi)} \rrangle^{\bP^1,T}_{0,1}.
$$
So \eqref{eqn:central-charge} is equivalent to \eqref{eqn:laplace-lambda}.

It remains to prove \eqref{eqn:central-charge} for $\cL=\cO_{\bP^1}(l_1 p_1 + l_2 p_1)$, where $l_1+l_2\geq 0$.
We will express both hand sides of \eqref{eqn:central-charge} in terms of (modified) Bessel functions.
A brief review of Bessel functions is given in Appendix \ref{sec:bessel}. The equivariant quantum
differential equation of $\bP^1$ is related to the modified Bessel differential equation by a simple
transform (see Appendix \ref{sec:QDE}).

Let $\gamma_{l_1,l_2}$ be defined as in Appendix \ref{sec:bessel}.
\begin{eqnarray*}
&&  \int_{\SYZ(\cL)} e^{\frac{W_t^\w}{z}} dy
= \int_{\SYZ(\cL)}\exp\big(\frac{1}{z}(e^y + t^0 + q e^{-y} +\w_1 y +\w_2(t^1-y)) \big) dy \\
&=& e^{\frac{1}{z}(t^0 + \w_2 t^1)} \int_{\gamma_{\ell_1,\ell_2}}
\exp\big(\frac{1}{z}(e^{y-i\pi}+ qe^{i\pi-y} + (\w_1-\w_2)(y-\pi i))\big) dy \\
&=& (-1)^{\frac{\w_1-w_2}{z}} e^{\frac{t^0}{z} + \frac{\w_1+\w_2}{2z}t^1} \int_{\gamma_{\ell_1,\ell_2}}
\exp\big(-\frac{2\sqrt{q}}{z}\cosh(y-\frac{t^1}{2}) +\frac{\w_1-\w_2}{z}(y-\frac{t^1}{2})\big) dy \\
&=& (-1)^{\frac{\w_1-w_2}{z}} e^{\frac{t^0}{z} + \frac{\w_1+\w_2}{2z}t^1} \int_{\gamma_{\ell_1,\ell_2}}
\exp\big(-\frac{2\sqrt{q}}{z}\cosh(y) +\frac{\w_1-\w_2}{z} y \big) dy
\end{eqnarray*}
By Lemma \ref{integral-bessel},
\begin{eqnarray*}
&& \int_{\gamma_{l_1,l_2}}  \exp\big(-\frac{2\sqrt{q}}{z}\cosh(y) +\frac{\w_1-\w_2}{z} y \big) dy \\
&=& \frac{\pi}{\sin(\frac{\w_2-\w_1}{z}\pi)} \big(e^{-2\pi i l_1 \frac{\w_1-\w_2}{z}}I_{ \frac{\w_1-\w_2}{z} }(\frac{2\sqrt{q}}{z})
- e^{-2\pi i l_2 \frac{\w_2-\w_1}{z}}I_{\frac{\w_2-\w_1}{z}}(\frac{2\sqrt{q}}{z})\big)\\
&=& -\sum_{\alpha=1}^2  e^{-2\pi i l_\alpha \frac{\chi_\alpha}{z}}
\frac{\pi}{\sin(\frac{\chi^\alpha}{z} \pi)} I_{\frac{\chi^\alpha}{z}}(\frac{2\sqrt{q}}{z}).
\end{eqnarray*}
Therefore, the left hand side of \eqref{eqn:central-charge} is
$$
\int_{\SYZ(\cL)} e^{\frac{W_t^\w}{z}} dy = -
e^{\frac{t^0}{z} + \frac{\w_1+\w_2}{2z}t^1}
\sum_{\alpha=1}^2 e^{-(2l_\alpha-1)\pi i \frac{\chi_\alpha}{z}}
 \frac{\pi}{\sin(\frac{\chi^\alpha}{z} \pi)} I_{\frac{\chi^\alpha}{z}}(\frac{2\sqrt{q}}{z}).
$$

Recall from Section \ref{sec:A-S} that
$$
J^\alpha =\llangle 1, \frac{\phi^\alpha}{z-\psi}\rrangle^{\bP^1,T}_{0,2}
=\chi^\alpha \llangle 1, \frac{\phi_\alpha}{z-\psi}\rrangle^{\bP^1,T}_{0,2}.
$$
We have
\begin{eqnarray*}
J^\alpha &=& e^{(t^0+t^1\w_\alpha)/z}\sum_{d=0}^\infty\frac{q^d}{d! z^d} \frac{1}{\prod_{m=1}^d(\chi^\alpha+mz)}\\
&=&  e^{(t^0+t^1\w_\alpha)/z} \sum_{m=0}^\infty(\frac{2\sqrt{q}}{z})^{2m}
\frac{\Gamma(\frac{\chi^\alpha}{z}+1)}{m! \Gamma(\frac{\chi^\alpha}{z}+m+1)}\\
&=& e^{\frac{t^0}{z} +\frac{\w^1+\w^2}{2z} t^1} z^{\frac{\chi^\alpha}{z}}
\Gamma(\frac{\chi^\alpha}{z}+1)I_{\frac{\chi^\alpha}{z}}(\frac{2\sqrt{q}}{z})
\end{eqnarray*}

\begin{eqnarray*}
\kappa(\cL) &=& \sum_{\alpha=1}^2
(-z)^{\frac{\chi^\alpha}{-z}+1}\Gamma(1-\frac{\chi^\alpha}{z})e^{\frac{-2l_\alpha \pi\sqrt{-1}\chi^\alpha}{z}} \phi_\alpha.
\end{eqnarray*}
So the right hand side of \eqref{eqn:central-charge} is
\begin{eqnarray*}
\llangle 1, \frac{\kappa(\cL)}{z-\psi} \rrangle_{0,2}^{\bP^1,T}
&=& \sum_{\alpha=1}^2
(-z)^{\frac{\chi^\alpha}{-z}+1}\Gamma(1-\frac{\chi^\alpha}{z})e^{\frac{-2\pi i l_\alpha \chi^\alpha}{z}} \frac{J^\alpha}{\chi^\alpha} \\
&=& -e^{\frac{t^0}{z} +\frac{\w^1+\w^2}{2z} t^1}
\sum_{\alpha=1}^2 (-1)^{\frac{\chi^\alpha}{-z}} e^{\frac{-2\pi i l_\alpha \chi^\alpha}{z}}\frac{\pi}{\sin(\frac{\chi^\alpha}{z}\pi)}
I_{\frac{\chi^\alpha}{z}}(\frac{2\sqrt{q}}{z})\\
&=& -e^{\frac{t^0}{z} +\frac{\w^1+\w^2}{2z} t^1}
\sum_{\alpha=1}^2  e^{-(2 l_\alpha-1)\pi i \frac{\chi^\alpha}{z}}
\frac{\pi}{\sin(\frac{\chi^\alpha}{z}\pi )} I_{\frac{\chi^\alpha}{z}}(\frac{2\sqrt{q}}{z})
\end{eqnarray*}

\end{proof}

\begin{remark}
Definition \ref{K-framing} (equivariant K-theoretic framing) and Definition \ref{SYZ-dual} (equivariant SYZ T-dual) can
be extended to any projective toric manifold. In \cite{FLZ14}, the first author uses the mirror theorem \cite{G96, LLY}
and results in \cite{Ir09} to extend Proposition \ref{central-charge} to any semi-Fano projective toric manifold.
The left hand side of \eqref{eqn:central-charge} is known as the central charge of the Lagrangian brane $\SYZ(\cL)$.
\end{remark}


\begin{proposition}\label{ABR}
The A and B-model $R$-matrices are equal
$$
R_\beta^{\ \alpha}(z)= \check R_\beta^{\ \alpha}(z).
$$
\end{proposition}
\begin{proof}
By the asymptotic decomposition theorem of the $S$-matrix (Theorem \ref{thm:factorize-S}),we only have to compare at the limit $q=0, t_0=0$ since both $\tilde S$ and $\check S$ are unitary. Notice that $\Psi$ has an non-degenerate limit at $q=0$, then it suffices to show that
$$\tilde S_i^{\ \halpha} e^{-u^\alpha/z}|_{q=0,t_0=0}\sim\frac{1}{\sqrt{-2\pi z}}\check S_i^{\ \balpha} e^{-\check u^\alpha/z}\vert_{q=0,t_0=0}.$$
The Lefschetz thimble $\gamma_2$ is $\{Y|Y\in (-\infty, 0)\}$. While the Lefschetz thimble $\gamma_1$ could not be explicitly depicted, we could alternatively consider the thimble $\gamma'_1=\{Y|Y\in (0,\infty)\}$ for $z<0$ of the oscillating integral $\int e^{W^\w_t/z} dy$. The integral yields the same asymptotic answer once we analytically continue $z<0$ to $z>0$, since the stationary phase expansion only depends on the local behavior (higher order derivatives) of $W^\w_t$ at the critical points.

So letting $Y=-Tz$ for $\alpha=2$, or $Y=-\frac{q}{Tz}$ for $\alpha=1$,
\begin{align*}
e^{-\check u^\alpha/z} \check S_0^{\ \balpha}=e^{-\frac{\Delta^\alpha(q)}{z}}(\frac{\chi^\alpha+\Delta^\alpha(q)}{2})^{\frac{\chi^\alpha}{z}}(-z)^{\frac{-\chi^\alpha}{z}}\int_{0}^\infty e^{-T}e^{-\frac{q}{Tz^2}} T^{\frac{\chi_\alpha}{z}-1} dT.
\end{align*}
Taking the limit $q\to 0$
\begin{align*}
\frac{1}{\sqrt{-2\pi z}} e^{-\check u^\alpha/z} \check S_0^{\ \balpha}\vert_{q=0}&=\frac{1}{\sqrt{-2\pi z}}e^{\frac{-\chi^\alpha}{z}} (\frac{-\chi^\alpha}{z})^{\frac{\chi^\alpha}{z}}\Gamma(\frac{-\chi^\alpha}{z})\\
&\sim \sqrt{\frac{1}{\chi^\alpha}} \exp(-\sum_{n=1}^\infty \frac{B_{2n}}{2n(2n-1)} (\frac{z}{\chi^\alpha})^{2n-1})\sim \tilde S_0^{\ \halpha} e^{-u^\alpha/z}\vert_{q=0}.
\end{align*}
Here we use the Stirling formula
$$
\log \Gamma(z) \sim \frac{1}{2}\log(2\pi) + (z-\frac{1}{2}) \log z-z+\sum_{n=1}^\infty \frac{B_{2n}}{2n(2n-1)}z^{1-2n}.
$$
Notice that $$\check S^{\ \balpha}_1=z \frac{\partial}{\partial t_1} \check S^{\ \alpha}_0=z \int_{\gamma_\alpha}e^{W^\w_t/z}(\frac{q}{Y}+\w_2) \frac{dY}{Y},$$
and similar calculation shows (letting $Y=-Tz$ if $\alpha=2$ and $Y=-\frac{q}{Tz}$ if $\alpha=1$)
$$
\frac{1}{\sqrt{-2\pi z}} e^{-\check u^\alpha/z} \check S_1^{\ \balpha}\vert_{q=0} \sim \w^\alpha  \sqrt{\frac{1}{\chi^\alpha}} \exp(-\sum_{n=1}^\infty \frac{B_{2n}}{2n(2n-1)} (\frac{z}{\chi^\alpha})^{2n-1})\sim \tilde S_1^{\ \halpha} e^{-u^\alpha/z}\vert_{q=0}.
$$
\end{proof}
Notice that the matrix $\check R$ is given by the asymptotic expansion. This theorem does not imply $\tilde S_i^{\ \halpha} e^{-u^\alpha/z}=\frac{1}{\sqrt{-2\pi z}}\check S_i^{\ \balpha} e^{-\check u^\alpha/z}$, which are unequal.

\subsection{The Eynard-Orantin topological recursion and the B-model graph sum}

Let $\omega_{g,n}$ be defined recursively by the Eynard-Orantin topological recursion \cite{EO07}:
$$
\omega_{0,1}=0,\quad  \omega_{0,2}=B(Y_1,Y_2)=\frac{dY_1\otimes dY_2}{(Y_1-Y_2)^2}.
$$
When $2g-2+n>0$,
\begin{eqnarray*}
\omega_{g,n}(Y_1,\ldots, Y_n) &=& \sum_{\alpha=1}^2\Res_{Y \to P_\alpha}
\frac{-\int_{\xi = Y}^{\hat{Y}} B(Y_n,\xi)}{2(\log(Y)-\log(\hat{Y}))dW}
\Big( \omega_{g-1,n+1}(Y,\hat{Y},Y_1,\ldots, Y_{n-1}) \\
&&\quad\quad  + \sum_{g_1+g_2=g}
\sum_{ \substack{ I\cup J=\{1,..., n-1\} \\ I\cap J =\emptyset } } \omega_{g_1,|I|+1} (Y,Y_I)\omega_{g_2,|J|+1}(\hat{Y},Y_J)
\end{eqnarray*}
where $Y\neq P_\alpha$ is in a small neighborhood of $P_\alpha$, and
$\hat{Y}\neq Y$ is the other point in the neighborhood such that $W_q^\w(\hat{Y})=W_q^\w(Y)$.

The B-model invariants $\omega_{g,n}$ can be expressed as graph sums \cite{KO, E11, E14, DOSS}.
We will use the formula stated in \cite[Theorem 3.7]{DOSS},  which is
equivalent to the formula in \cite[Theorem 5.1]{E11}.
Given a labeled graph $\vGa \in \bGa_{g,n}(\bP^1)$ with $L^o(\Ga)=\{l_1,\ldots,l_n\}$, we define its weight to be
\begin{eqnarray*}
w(\vGa)&=& (-1)^{g(\vGa)-1+n}\prod_{v\in V(\Gamma)} \Big(\frac{h^\alpha_1}{\sqrt{2}}\Big)^{2-2g-\val(v)} \langle \prod_{h\in H(v)} \tau_{k(h)}\rangle_{g(v)}
\prod_{e\in E(\Gamma)} \check{B}^{\alpha(v_1(e)),\alpha(v_2(e))}_{k(e),l(e)} \\
&& \cdot \prod_{j=1}^n \frac{1}{\sqrt{-2}} d\xi^{\alpha(l_j)}_{k(l_j)}(Y_j)
\prod_{l\in \cL^1(\Gamma)}(-\frac{1}{\sqrt{-2}})\check{h}^{\alpha(l)}_{k(l)}.
\end{eqnarray*}
Here, $\check{h}^\alpha_k=-\frac{1}{\sqrt{-1}^{2k-1}}2(2k-1)!!h^\alpha_{2k-1}$. Note that the definitions
of $\check{B}^{\alpha,\beta}_{k,l}$, $\check{h}_k^\alpha$, 
$d\xi^\alpha_k$ in this paper
are slightly different from those in \cite{DOSS}; for example,
the definition of $\check{B}_{k,l}^{\alpha,\beta}$ in this paper
differs from  Equation (3.11) of \cite{DOSS} by a factor of $2^{-k-l-1}$.
In our notation \cite[Theorem 3.7]{DOSS} is equivalent to:
\begin{theorem} \label{thm:DOSS}
For $2g-2+n>0$,
$$
\omega_{g,n} = \sum_{\Gamma \in \bGa_{g,n}(\bP^1)}\frac{w(\vGa)}{|\Aut(\vGa)|}.
$$
\end{theorem}

\subsection{All genus mirror symmetry}\label{sec:main-results}
Given a meromorphic function $f(Y)$ on $\bP^1$ which is holomorphic
on $\bP^1\setminus \{ P_1, P_2\}$,  define
$$
\theta(f) = \frac{df}{dW} = \frac{Y^2}{(Y-P_1)(Y-P_2)} \frac{df}{dY}.
$$
Then $\theta(f)$ is also a meromorphic function which is holomorphic on $\bP^1\setminus \{P_1,P_2\}$.
For $\alpha\in \{1,2\}$, let
$$
\xi_{\alpha,0} = \frac{1}{\sqrt{-1}}\sqrt{\frac{2}{\Delta^\alpha(q)}} \frac{P_\alpha}{Y-P_\alpha}.
$$
Then $\xi_{\alpha,0}$ is a meromorphic function on $\bP^1$ with a simple
pole at $Y=P_\alpha$ and holomorphic elsewhere. Moreover,
the differential of $\xi_{\alpha,0}$ is $d\xi_{\alpha,0}$. For $k>0$, define
$$
W^\alpha_k := d((-1)^k\theta^k(\xi_{\alpha,0})).
$$
Define
\begin{align*}
\check S^{\ \balpha}_{\hubeta}(z)=-z \int_{y\in \gamma_\alpha} e^{\frac{x}{z}} \frac{d\xi_{\beta,0}}{\sqrt{-2}},
\quad \check S^{\ \kappa(\cL)}_{\hubeta}(z)=-z \int_{y\in \mathrm{SYZ}(\cL)} e^{\frac{x}{z}} \frac{d\xi_{\beta,0}}{\sqrt{-2}}.
\end{align*}
Then
\begin{align}\label{eqn:W-laplace}
\check S^{\ \balpha}_{\hubeta}(z)=-z^{k+1}\int_{y\in \gamma_\alpha} e^{\frac{W(y)}{z}} \frac{W^\beta_k}{\sqrt{-2}},
\quad \check S^{\ \kappa(\cL)}_{\hubeta}(z)=-z^{k+1}\int_{y\in \mathrm{SYZ}(\cL)} e^{\frac{W(y)}{z}} \frac{W^\beta_k}{\sqrt{-2}}
\end{align}
Therefore,
\begin{equation}\label{eqn:W-S}
\int_{y\in \mathrm{SYZ}(\cL)} e^{\frac{W(y)}{z}} \frac{W^\beta_k}{\sqrt{-2}} = -z^{-k-1}
\check S^{\ \kappa(\cL)}_{\hubeta}(z) = -z^{-k-1} \llangle \hat{\phi}_\alpha(q), \frac{\kappa(\cL)}{z-\psi}\rrangle^{\bP^1,T}_{0,2}.
\end{equation}
where the last equality follows from Theorem \ref{central-charge}.

For $\alpha=1,2$ and $j=1,\cdots,n$, let
\begin{equation}
\label{eqn:utilde}
\tilde{\bu}_j^\alpha(z) =\sum_{\beta=1}^2 S^\hualpha_{\,\ \hubeta}(z) \frac{\bu^\beta_j(z)}{\sqrt{\Delta^\beta(q)}}.
\end{equation}

\begin{Theorem}[All genus equivariant mirror symmetry for $\bP^1$]\label{main}
For $n>0$ and $2g-2+n>0$, we have
\begin{equation}
\omega_{g,n}\big|_{\frac{1}{\sqrt{-2}}W^\alpha_k(Y_j)=(\tilde{u}_j)^\alpha_k} = (-1)^{g-1+n}F_{g,n}^{\bP^1,T}(\bu_1,\cdots,\bu_n,\bt).
\end{equation}
\end{Theorem}
\begin{proof}
We will prove this theorem by comparing the A-model graph sum in the end of Section 2.7 and the B-model graph sum in the previous section.
\begin{enumerate}
\item Vertex. By Section 3.1, we have $h^\alpha_1(q)=\sqrt{\frac{2}{\Delta^\alpha(q)}}$. So in the B-model vertex, $\frac{h^\alpha_1}{\sqrt{2}}=\sqrt{\frac{1}{\Delta^\alpha(q)}}$. Therefore the B-model vertex matches the A-model vertex.

\item  Edge. By Section 3.6, we know that
\begin{eqnarray*}
\check{B}^{\alpha,\beta}_{k,l}&=&[u^{-k}v^{-l}]\left(\frac{uv}{u+v}(\delta_{\alpha,\beta}
-\sum_{\gamma=1,2} f^\alpha_\gamma(u,q)f^\beta_\gamma(v,q))\right)\\
&=&[z^{k}w^{l}]\left(\frac{1}{z+w}(\delta_{\alpha,\beta}
-\sum_{\gamma=1,2} f^\alpha_\gamma(\frac{1}{z},q)f^\beta_\gamma(\frac{1}{w},q))\right).
\end{eqnarray*}
By definition
$$
\cE^{\alpha,\beta}_{k,l} = [z^k w^l]
\Bigl(\frac{1}{z+w} (\delta_{\alpha,\beta}-\sum_{\gamma=1,2} R_\gamma^{\,\ \alpha}(-z) R_\gamma^{\,\ \beta}(-w)\Bigr).
$$
But we know that
$$R_\beta^{\,\ \alpha}(z) = f^\alpha_\beta(\frac{-1}{z}).$$
Therefore, we have
$$
\check{B}^{\alpha,\beta}_{k,l} = \cE^{\alpha,\beta}_{k,l}.
$$

\item Ordinary leaf. We have the following expression for $d\xi^{\alpha}_{k}$ (see \cite{FLZ}):
$$
d\xi^{\alpha}_{k}=W^\alpha_k -\sum_{i=0}^{k-1}\sum_{\beta}
\check{B}^{\alpha,\beta}_{k-1-i,0}W^\beta_i.
$$
By item 2 (Edge) above, for $k,l\in \bZ_{\geq 0}$,
\begin{eqnarray*}
\check{B}^{\alpha,\beta}_{k,l}&=&
[z^k w^l]
\Bigl(\frac{1}{z+w} (\delta_{\alpha,\beta}-\sum_{\gamma=1,2} R_\gamma^{\,\ \alpha}(-z) R_\gamma^{\,\ \beta}(-w))\Bigr).
\end{eqnarray*}
We also have
$$
[z^0](R_\beta^{\,\ \alpha}(-z))=\delta_{\alpha,\beta}.
$$
Therefore,
$$
d\xi^{\alpha}_{k}=\sum_{i=0}^{k}\sum_{\beta=1}^2\left([z^{k-i}]R_\beta^{\,\ \alpha}(-z)\right)W^\beta_i.
$$
So under the identification
$$
\frac{1}{\sqrt{-2}}W^\alpha_k(Y_j)=(\tilde{u}_j)^\alpha_k
$$
The B-model ordinary leaf matches the A-model ordinary leaf.

\item Dilaton leaf. We have the following relation between $\check{h}^{\alpha}_{k}$ and $f^\alpha_\beta(u,q)$ (see \cite{FLZ})
$$\check{h}^{\alpha}_{k}=[u^{1-k}]\sum_{\beta}\sqrt{-1}h^\beta_1f^\alpha_\beta(u,q).$$
By the relation
$$
R_\beta^{\,\ \alpha}(z) = f^\alpha_\beta(\frac{-1}{z})
$$
and the fact $h^\beta_1(q)=\sqrt{\frac{2}{\Delta^\beta(q)}}$, it is easy to see that the B-model dilaton leaf matches the A-model dilaton leaf.
\end{enumerate}
\end{proof}

Taking Laplace transforms at appropriate cycles to Theorem \ref{main}
produces a theorem concerning descendants potential.
\begin{Theorem}[All genus full descendant equivariant mirror symmetry for $\bP^1$] \label{full-descendant}
Suppose that $n>0$ and $2g-2+n>0$. For any $\cL_1,\ldots,\cL_n\in K_T(\bP^1)$,
there is a formal power series identity
\begin{equation} \label{eqn:psi}
\begin{aligned}
& \int_{y_1\in \SYZ(\cL_1)}\cdots\int_{y_n\in \SYZ(\cL_n)}
e^{\frac{W(y_1)}{z_1}+\cdots+ \frac{W(y_n)}{z_n}}\omega_{g,n}\\
= &(-1)^{g-1}     \llangle \frac{\kappa(\cL_1)}{z_1-\psi_1},\dots,\frac{\kappa(\cL_n)}{z_n-\psi_n} \rrangle_{g,n}.
\end{aligned}
\end{equation}
\end{Theorem}

\begin{remark}
By Theorem \ref{central-charge},
\begin{equation}\label{eqn:psi-disk}
\int_{y_1\in \SYZ(\cL)} e^{\frac{W(y_1)}{z_1}}ydx =  - \llangle \frac{\kappa(\cL_1)}{z_1-\psi_1}\rrangle_{0,1}^{\bP^1,T}
\end{equation}
which is the analogue of  \eqref{eqn:psi} in the unstable  case $(g,n)=(0,1)$.
\end{remark}

\begin{proof}[Proof of Theorem \ref{full-descendant}]
By \eqref{eqn:utilde},
$$
\tilde \bu^\alpha_j(z)=\sum_{\beta=1}^2 \sqrt{\Delta^\alpha(q)} \llangle \phi_\alpha(q),\frac{\phi_\beta(q)}{z-\psi} \rrangle^{\bP^1,T}_{0,2}  \bu^\beta_j(z).
$$
Define the flat coordinates $\overline \bu_j^\alpha$ by
$$
\sum_{\alpha=1}^2 \bu^\alpha_j(z) \phi_\alpha(q)= \sum_{\alpha=1}^2 \overline \bu_j^\alpha(z) \phi_\alpha(0),
$$
and a power series in $1/z$
$$
S^{\hualpha}_{\ \beta}(z)=\llangle \hat \phi_\alpha(q),\frac{\phi_\beta(0)}{z-\psi}\rrangle_{0,2}.
$$
Then
\begin{align*}
\tilde \bu^\alpha_j(z)=\sum_{\beta=1}^2 \Bigl (\llangle \hat\phi_\alpha(q),\frac{\phi_\beta(0)}{z-\psi} \rrangle  \overline \bu^\beta_j(z)\Bigr)_+
=\sum_{\beta=1}^2 \bigl( S^{\hualpha}_{\ \beta}(z) \overline \bu^\beta_j(z) \bigr)_+.
\end{align*}
Notice that $(S^{\hualpha}_{\ \beta})$ is unitary, i.e. $\sum_{\gamma} S^{\hugamma}_{\ \alpha}(z) S^\hugamma_{\ \beta}(-z)=\frac{1}{\chi^\beta}\delta_{\alpha\beta}$.
We have
\begin{align*}
\sum_{\alpha=1}^2 \bigl( S^\hualpha_{\ \gamma}(-z) \tilde \bu^\alpha_j(z)\bigr)_+=\sum_{\alpha=1}^2 \Bigl( \sum_{\beta =1}^2 S^\hualpha_{\ \beta}(z) S^\hualpha_{\ \gamma}(-z) \overline \bu^\beta_j(z) \Bigr)=\frac{\overline \bu^\gamma_j(z)}{\chi^\gamma}.
\end{align*}
Taking the Laplace transform of $\omega_{g,n}$
\begin{align*}
&\int_{y_1\in \SYZ(\cL_1)}\dots \int_{y_n\in \SYZ(\cL_n)} e^{\frac{W(y_1)}{z_1}+\dots +\frac{W(y_n)}{z_n}}\omega_{g,n}\\
=&\int_{y_1\in \SYZ(\cL_1)}\dots\int_{y_n\in \SYZ(\cL_n)} e^{\sum_{i=1}^n \frac{W(y_i)}{z_i}} (-1)^{g-1+n}\bigl( \sum_{\beta_i,a_i} \llangle
\prod_{i=1}^n\tau_{a_i}(\phi_{\beta_i}(0)) \rrangle_{g,n} \\
& \quad\quad \cdot \prod_{i=1}^n (\overline u_i)^{\beta_i}_{a_i}\bigr)\vert_{{(\tilde u_j)}^\beta_k=\frac{1}{\sqrt{-2}}W^\beta_k(y_j)}\\
=&\int_{y_1\in \SYZ(\cL_1)}\dots\int_{y_n\in \SYZ(\cL_n)}
e^{\sum_{i=1}^n \frac{W(y_i)}{z_i}} (-1)^{g-1+n}\bigl( \sum_{\beta_i,a_i} \llangle \prod_{i=1}^n\tau_{a_i}(\phi_{\beta_i}(0)) \rrangle_{g,n} \\
&\quad\quad \cdot \prod_{i=1}^n (\chi^{\beta_i}\sum_{\alpha=1}^2 \sum_{k\in \bZ_{\ge 0}}[z_i^{a_i-k}] S_{\ \beta_i}^{\hualpha}(-z_i)
\frac{W^\alpha_{k}(y_i)}{\sqrt{-2}}\bigr).
\end{align*}
Using \eqref{eqn:W-S}
\begin{align*}
&\int_{y_1\in \SYZ(\cL_1)}\dots \int_{y_n\in \SYZ(\cL_n)} e^{\frac{W(y_1)}{z_1}+\dots +\frac{W(y_n)}{z_n}}\omega_{g,n}\\
=& (-1)^{g-1+n} \bigl(  \sum_{\beta_i,a_i} \llangle \prod_{i=1}^n\tau_{a_i}(\phi_{\beta_i}(0)) \rrangle_{g,n} \prod_{i=1}^n (\chi^{\beta_i}\sum_{\alpha=1}^2\sum_{k\in \bZ_{\geq 0}}([z_i^{a_i-k}] S_{\ \beta_i}^\hualpha(-z_i))S^{\ \kappa(\cL_i)}_\hualpha (z_i)( -z_i^{-k-1}))\bigr)\\
=&(-1)^{g-1}      \sum_{\beta_i,a_i} \llangle \prod_{i=1}^n\tau_{a_i}(\phi_{\beta_i}(0)) \rrangle_{g,n}  \prod_{i=1}^n  \chi^{\beta_i} (\phi_{\beta_i}(0), \kappa(\cL_i))  z_i^{-a_i-1}\\
=&(-1)^{g-1}     \llangle \frac{\kappa(\cL_1)}{z_1-\psi_1},\dots,\frac{\kappa(\cL_n)}{z_n-\psi_n} \rrangle_{g,n}.
\end{align*}
\end{proof}

\section{The non-equivariant limit and the Norbury-Scott conjecture}

In this section, we consider the non-equivariant limit $\w_1=\w_2=0$.
\subsection{The non-equivariant $R$-matrix}
By \cite[Section 1.3]{G01}, $R(z)=I+\sum_{n=1}^\infty R_n z^n$ is
uniquely determined by:
\begin{enumerate}
\item The recursive relation: $(d+\Psi^{-1}d\Psi)R_n =[dU, R_{n+1}]$.
\item The homogeneity of $R(z)$: $R_nq^{n/2}$ is a constant matrix.
\end{enumerate}
The unique solution $R(z)$ satisfying the above conditions was computed
explicitly in \cite{SS}:
\begin{lemma}[ {\cite[Lemma 3.1]{SS}} ]
$$
R_n= q^{-\frac{n}{2}}\frac{(2n-1)!!(2n-3)!!}{n!2^{4n}}
\begin{pmatrix} -1 & 2n\sqrt{-1}(-1)^{n+1} \\ 2n\sqrt{-1}  & (-1)^{n+1}  \end{pmatrix}
$$
\end{lemma}
By Proposition \ref{ABR} , $R(z)=\check{R}(z)$. In this subsection, we recover
the above lemma by computing the stationary phase expansion of $\check{S}$.

We assume $z,q\in (0,\infty)$, where $q=Qe^{t^1}$.
\begin{eqnarray*}
\check{S}_0^{\,\ 2} &=& \int_{y=-\infty}^{y=+\infty} e^{\frac{1}{z}(t^0+ e^{y-i\pi} + qe^{-(y-i\pi)}) } dy \\
&=&e^{t^0/z}\int_{y=-\infty}^{y=+\infty} e^{-\frac{2\sqrt{q}}{z}\cosh(y-\frac{t^1}{2})} dy\\
&=& e^{t^0/z}\int_{y=-\infty}^{y=+\infty}  e^{-\frac{2\sqrt{q}}{z}\cosh(y)} dy \\
&=&2 e^{(t^0-2\sqrt{q})/z}\int_{y=0}^{y=+\infty} e^{-\frac{2\sqrt{q}}{z}(\cosh(y)-1)}dy.
\end{eqnarray*}
Let $T=\frac{2\sqrt{q}}{z}(\cosh(y)-1)$, then
$$
y=\cosh^{-1}(1+\frac{zT}{2\sqrt{q}}),\quad
dy=\frac{1}{2}q^{-\frac{1}{4}}T^{-1/2}\sqrt{ \frac{z}{ 1+\frac{zT}{4\sqrt{q}} } }.
$$
\begin{eqnarray*}
\check{S}_0^{\,\ 2}
&=& e^{(t^0-2\sqrt{q})/z}\sum_{n=0}^\infty
(\frac{z}{\sqrt{q}})^{n+\frac{1}{2}}{-1/2 \choose n} 2^{-2n}\int_{T=0}^{T=+\infty} e^{-T}  T^{n-1/2} dT\\
&=& e^{(t^0-2\sqrt{q})/z}\sum_{n=0}^\infty (\frac{z}{\sqrt{q}})^{n+\frac{1}{2}}
\frac{(-1)^n(2n-1)!!}{n! 2^{3n}} \Gamma(n+\frac{1}{2})\\
&=& \sqrt{\pi}e^{(t^0-2\sqrt{q})/z}\sum_{n=0}^\infty (\frac{z}{\sqrt{q}})^{n+\frac{1}{2}}
\frac{(-1)^n ((2n-1)!!)^2}{n! 2^{4n}};
\end{eqnarray*}
\begin{align*}
\check{S}_1^{\,\ 2}=z\frac{\partial}{\partial t_1} \check{S}_0^{\,\ 2}
= \sqrt{\pi}ze^{(t^0-2\sqrt{q})/z}\sum_{n=0}^\infty
(\frac{z}{\sqrt{q}})^{n-1/2} (1+(\frac{1}{4}+\frac{n}{2})\frac{z}{\sqrt{q}}  )
\frac{(-1)^{n+1}((2n-1)!!)^2}{n! 2^{4n}}.
\end{align*}

Similarly,
$$
\check{S}_0^{\,\ 1} = \sqrt{-\pi}e^{(t^0+2\sqrt{q})/z}\sum_{n=0}^\infty
(\frac{z}{\sqrt{q}})^{n+\frac{1}{2}}
\frac{((2n-1)!!)^2}{n! 2^{4n}};
$$
$$
\check{S}_1^{\,\ 1} = \sqrt{-\pi}ze^{(t^0+2\sqrt{q})/z}\sum_{n=0}^\infty
(\frac{z}{\sqrt{q}})^{n-\frac{1}{2}}(1-(\frac{1}{4}+\frac{n}{2})\frac{z}{\sqrt{q}})
\frac{((2n-1)!!)^2}{n! 2^{4n}}.
$$
Therefore,
$$
\tS(z) =\frac{1}{\sqrt{-2\pi z}}\check{S}(z),
$$
\begin{align*}
&[z^n] \left( \tS(z)e^{-U/z}\right) \\
=& \begin{pmatrix} \frac{((2n-1)!!)^2}{\sqrt{2}n!2^{4n}q^{\frac{n}{2}+\frac{1}{4}}}
& \frac{\sqrt{-1}(-1)^{n+1}((2n-1)!!)^2}{\sqrt{2}n!2^{4n}q^{\frac{n}{2}+\frac{1}{4}}}  \\
\frac{((2n-1)!!)^2}{\sqrt{2}n!2^{4n}q^{\frac{n}{2}-\frac{1}{4}}}
-(\frac{n}{2}-\frac{1}{4})\frac{((2n-3)!!)^2}{\sqrt{2}(n-1)! 2^{4n-4} q^{\frac{n}{2}-\frac{1}{4}}}
& \frac{\sqrt{-1}(-1)^n(2n-1)!!)^2}{\sqrt{2}n! 2^{4n} q^{\frac{n}{2}-\frac{1}{4}}}
+(\frac{n}{2}-\frac{1}{4})\frac{\sqrt{-1}(-1)^{n+1}((2n-3)!!)^2}{\sqrt{2}(n-1)! 2^{4n-4} q^{\frac{n}{2}-\frac{1}{4}}} \end{pmatrix},
\end{align*}
\begin{align*}
R_n  &= \begin{pmatrix}-\frac{(2n-1)!!(2n-3)!!}{n!2^{4n}} &
\frac{\sqrt{-1}(-1)^{n+1}(2n-1)!!(2n-3)!!}{(n-1)!2^{4n-1}} \\
\frac{\sqrt{-1}(2n-1)!!(2n-3)!!}{(n-1)!2^{4n-1}} &
\frac{(-1)^{n+1}(2n-1)!!(2n-3)!!}{n!2^{4n}} \end{pmatrix}q^{-\frac{n}{2}} \\
&= q^{-\frac{n}{2}}\frac{(2n-1)!!(2n-3)!!}{n!2^{4n}}
\begin{pmatrix} -1 & 2n\sqrt{-1}(-1)^{n+1} \\ 2n\sqrt{-1}  & (-1)^{n+1}  \end{pmatrix}
\end{align*}

\subsection{The Norbury-Scott Conjecture} \label{sec:NS}
In this subsection, we assume $\w_1=\w_2=t^0=0$. Then
$$
\llangle \tau_{a_1}(H) \cdots \tau_{a_n}(H)\rrangle^{\bP^1}_{g,n}=
q^{\frac{1}{2}(\sum_{i=1}^n a_i) +1-g} \langle \tau_{a_1}(H) \cdots \tau_{a_n}(H)\rangle^{\bP^1}_{g,n}.
$$
Note that when $\frac{1}{2}(\sum_{i=1}^n a_i)+ 1-g$ is not an nonnegative integer, both hand sides are zero.

When $2g-2+n>0$, $\omega_{g,n}$ is holomorphic near $Y=0$, and one may expand it in the local holomorphic
coordinate $\tx = x^{-1}= (Y+\frac{q}{Y})^{-1}$.
\begin{theorem} Suppose that $2g-2+n>0$. Then  near $Y=0$, $\omega_{g,n}$ has the following
expansion
$$
\omega_{g,n}= (-1)^{g-1+n} \sum_{a_1,\ldots, a_n\in \bZ_{\geq 0}}\llangle \tau_{a_1}(H)\cdots \tau_{a_n}(H)\rrangle_{g,n}^{\bP^1}\prod_{j=1}^n
\frac{(a_j+1)!}{x^{a_j+2}}dx_j
$$
\end{theorem}
The Norbury-Scott conjecture corresponds to the specialization $q=1$, i.e. $t^1=0,Q=1$.

\begin{proof}
Define $\tW^\alpha_k$ by
$$
\frac{1}{\sqrt{-2}} \tW^\alpha_k = \tilde{\bu}^\alpha_k\Big|_{t^0_a=0, t^1_a= (a+1)!x^{-a-2}dx}.
$$
By Theorem \ref{main}, it suffices to show that $\tW^\alpha_k$ agrees with
the expansion of $W^\alpha_k$ near $Y=0$ in $\tx = x^{-1}$.

We now compute $\tW^\alpha_k$ explicitly.
\begin{eqnarray*}
J &=& e^{(t^0+t^1 H)/z} (1+ \sum_{d=1}^\infty \frac{q^d}{\prod_{m=1}^d(H+mz)^2})\\
&=& e^{\frac{t^0}{z}}(1+t^1\frac{H}{z})\Big(1+\sum_{d=1}^\infty\frac{q^d}{z^{2d}(d!)^2}-2 (\sum_{d=1}^\infty\frac{q^d}{z^{2d}(d!)^2} \sum_{m=1}^d \frac{1}{m})\frac{H}{z})\Big) \\
&=& e^{\frac{t^0}{z}} (1+\sum_{d=1}^\infty\frac{q^d}{z^{2d}(d!)^2}) \\
&&  + e^{\frac{t^0}{z}}(t^1 (1+\sum_{d=1}^\infty\frac{q^d}{z^{2d+1}(d!)^2}) -2 \sum_{d=1}^\infty\frac{q^d}{z^{2d+1}(d!)^2}\sum_{m=1}^d\frac{1}{m}) H
\end{eqnarray*}
\begin{eqnarray*}
z\frac{\partial J}{\partial t^1} &=& e^{\frac{t^0}{z}} (\sum_{d=1}^\infty\frac{dq^d}{z^{2d-1}(d!)^2}) \\
&&  + e^{\frac{t^0}{z}}\big(t^1(\sum_{d=1}^\infty  \frac{dq^d}{z^{2d}(d!)^2}) +
1 + \sum_{d=1}^\infty\frac{q^d}{z^{2d}(d!)^2}(1-2d \sum_{m=1}^d\frac{1}{m})\big)H
\end{eqnarray*}
\begin{align*}
S^0_{\,\ 0}(z)=& (H,\cS(1))=(1,z\frac{\partial J}{\partial t^1})=
e^{\frac{t^0}{z}}\big(t^1(\sum_{d=1}^\infty  \frac{dq^d}{z^{2d}(d!)^2}) +
1 + \sum_{d=1}^\infty\frac{q^d}{z^{2d}(d!)^2}(1-2d \sum_{m=1}^d\frac{1}{m})\big)\\
S^1_{\,\ 0}(z)=& (1,\cS(1))=(1,J)=  e^{\frac{t^0}{z}}(t^1 (1+\sum_{d=1}^\infty\frac{q^d}{z^{2d+1}(d!)^2}) -2 \sum_{d=1}^\infty\frac{q^d}{z^{2d+1}(d!)^2}\sum_{m=1}^d\frac{1}{m})\\
S^0_{\,\ 1}(z)=& (H,\cS(H))=(H,z\frac{\partial J}{\partial t^1})=e^{\frac{t^0}{z}} (\sum_{d=0}^\infty\frac{q^{d+1}}{z^{2d+1}d!(d+1)!})\\
S^1_{\,\ 1}(z)=& (1,\cS(H))=(H,J)=  e^{\frac{t^0}{z}} (1+\sum_{d=1}^\infty\frac{q^d}{z^{2d}(d!)^2})
\end{align*}
$$
S^{\underline{\hat{\alpha}}}_{\,\ j}(z)=\sum_{i=0}^1 \Psi_i^{\,\ \alpha}S^i_{\,\ j}(z)
$$
\begin{align*}
S^{\underline{\hat{1}}}_{\,\ 1}(z)=&\frac{1}{\sqrt{2}}e^{\frac{t^0}{z}}\sum_{n=0}^\infty \frac{(\sqrt{q})^{n+\frac{1}{2}}}{z^n}
\frac{1}{\lfloor \frac{n}{2}\rfloor ! \lceil\frac{n}{2}\rceil!} \\
S^{\underline{\hat{2}}}_{\,\ 1}(z)=& \frac{1}{\sqrt{2}}e^{\frac{t^0}{z}}\sum_{n=0}^\infty \frac{(-\sqrt{q})^{n+\frac{1}{2}}}{z^n}
\frac{1}{\lfloor \frac{n}{2}\rfloor ! \lceil\frac{n}{2}\rceil!}\\
\end{align*}
$$
\tilde{\bu}^{\alpha}(z)=\sum_{i=0}^1 S^{\underline{\hat{\alpha}}}_{\,\ i}(z) t^i(z)
$$
\begin{eqnarray*}
\tilde{u}_k^{1}\Big|_{t^0_a=0} &=& \frac{1}{\sqrt{2}}\sum_{n=0}^\infty
\frac{(\sqrt{q})^{n+\frac{1}{2}}}{\lfloor \frac{n}{2}\rfloor ! \lceil\frac{n}{2}\rceil!} t^1_{k+n},\\
\tilde{u}_k^{2}\Big|_{t^0_a=0} &=&\frac{1}{\sqrt{2}}\sum_{n=0}^\infty
\frac{(-\sqrt{q})^{n+\frac{1}{2}}}{\lfloor \frac{n}{2}\rfloor ! \lceil\frac{n}{2}\rceil!} t^1_{k+n}
\end{eqnarray*}
For $\alpha=1,2$,
\begin{equation}\label{eqn:tWxi}
\tW^\alpha_k = \sqrt{-2} \tilde{u}_k^\alpha\Big|_{t^0_a=0, t^1_a=  (a+1)! x^{-a-2}dx}
=  d((-\frac{d}{dx})^k \txi_{\alpha,0})
\end{equation}
where
\begin{equation}\label{eqn:txi-one}
\txi_{1,0} :=  - \frac{1}{\sqrt{-1}}\sum_{n=0}^\infty
(\sqrt{q})^{n+\frac{1}{2}} {n \choose \lfloor \frac{n}{2}\rfloor}x^{-n-1}
\end{equation}
\begin{equation}\label{eqn:txi-two}
\txi_{2,0} :=  - \frac{1}{\sqrt{-1}}\sum_{n=0}^\infty
(-\sqrt{q})^{n+\frac{1}{2}} {n \choose \lfloor \frac{n}{2}\rfloor} x^{-n-1}
\end{equation}

Recall that
\begin{equation}\label{eqn:Wxi}
W^\alpha_k = d((-\frac{d}{dx})^k \xi_{\alpha,0})
\end{equation}
By \eqref{eqn:tWxi} and  \eqref{eqn:Wxi},
to complete the proof, it remains to show that,
$\txi_{\alpha,0}$ agree with the expansion
of $\xi_{\alpha,0}$ near $Y=0$ in $\tx=x^{-1}= (Y+\frac{q}{Y})^{-1}$.


Assume that $q \in (0,\infty)$. We have
\begin{eqnarray*}
&& P_1=\sqrt{q},\quad \Delta^1 = 2\sqrt{q},\quad \xi_{1,0}= \frac{1}{\sqrt{-1}}\frac{q^{1/4}}{Y-\sqrt{q}},\\
&& P_2=-\sqrt{q},\quad \Delta^2 = -2\sqrt{q},\quad \xi_{2,0}=\frac{q^{1/4}}{Y+\sqrt{q}},\\
\end{eqnarray*}

The $n$-th coefficient in the expansion of $\tx =(Y+\frac{q}{Y})^{-1}$ at $Y=0$ is given by the residue
\begin{align*}
\mathrm{Res}_{Y=0} \tx^{-n-1} \xi_{1,0} d\tx
=&- \frac{1}{\sqrt{-1}}q^{1/4}\mathrm{Res}_{Y=0} (Y+\frac{q}{Y})^{n-1} (1-\frac{q}{Y^2})\frac{dY}{Y-\sqrt{q}} \\
=& - \frac{1}{\sqrt{-1}}q^{1/4}\mathrm{Res}_{Y=0} \frac{(Y^2+q)^{n-1} (Y+\sqrt{q})}{Y^{n+1}}dY \\
=& -\frac{1}{\sqrt{-1}}(\sqrt{q})^{n-\frac{1}{2}}{n-1 \choose \lfloor \frac{n}{2}\rfloor}
\end{align*}
$$
\xi_{1,0} =  -\frac{1}{\sqrt{-1}}\sum_{n=1}^\infty (\sqrt{q})^{n-\frac{1}{2}}{n-1 \choose \lfloor \frac{n}{2}\rfloor}\tx^n
= -\frac{1}{\sqrt{-1}}\sum_{n=0}^\infty (\sqrt{q})^{n+\frac{1}{2}}{n\choose \lfloor \frac{n+1}{2} \rfloor} \tx^{n+1}
= -\frac{1}{\sqrt{-1}}\sum_{n=0}^\infty (\sqrt{q})^{n+\frac{1}{2}}{n\choose \lfloor \frac{n}{2} \rfloor} x^{-n-1}
$$
which agrees with $\txi_{1,0}$ defined in \eqref{eqn:txi-one}.

\begin{align*}
\mathrm{Res}_{Y=0} \tx^{-n-1} \xi_{2,0} d\tx
=&- q^{1/4}\mathrm{Res}_{Y=0} (Y+\frac{q}{Y})^{n-1} (1-\frac{q}{Y^2})\frac{dY}{Y+\sqrt{q}} \\
=& - q^{1/4}\mathrm{Res}_{Y=0} \frac{(Y^2+q)^{n-1} (Y-\sqrt{q})}{Y^{n+1}}dY \\
=& -\frac{1}{\sqrt{-1}}(-\sqrt{q})^{n-\frac{1}{2}}{n-1\choose \lfloor\frac{n}{2} \rfloor}
\end{align*}
$$
\xi_{2,0}= -\frac{1}{\sqrt{-1}}\sum_{n=0}^\infty (-\sqrt{q})^{n+\frac{1}{2}}{n\choose \lfloor \frac{n}{2} \rfloor} x^{-n-1}
$$
which agrees with $\txi_{2,0}$ defined in \eqref{eqn:txi-two}.

\end{proof}

\section{The large radius limit and the Bouchard-Mari\~{n}o conjecture}
\label{sec:BM}

In this section, we will specialize Theorem \ref{main} to the large radius limit case. In this case, Theorem \ref{main} relates the invariant $\omega_{g,n}$ of the limit curve to the equivariant descendent theory of $\bC$. After expanding $\xi_{\alpha,0}$ in suitable coordinates, we can relate the corresponding expansion of $\omega_{g,n}$ to the generation function of Hurwitz numbers and therefore reprove the Bouchard-Mari\~{n}o conjecture \cite{BM08} on Hurwitz numbers.

Let $\w_2=0$, $t_0=0$ and take the large radius limit $q\to 0$. Then our mirror curve becomes
$$x=Y+\w_1\log Y.$$
When $\w_1=-1$, this is just the Lambert curve.
Recall that the two critical points $P_1,P_2$ of $W^{\w}_t(Y)$ are
$$P_{\alpha}=\frac{\w_2-\w_1+\Delta^\alpha(q)}{2}.$$
Since $\Delta^1(0)=\w_1-\w_2$, $P_1\to 0$ under the limit $q\to 0$. In other words, $P_1$ goes out of the curve under the limit $q\to 0$ and $\xi_{1,0}= \sqrt{\frac{2}{\Delta^\alpha(q)}} \frac{P_1}{Y-P_1}\to 0$. As a result, $W^1_k = d(\theta^k(\xi_{1,0}))$ also turns to zero under the large radius limit.

Under the identification $\frac{1}{\sqrt{-2}}W^\alpha_k(Y_j)=(\tilde{u}_j)^\alpha_k$ in Theorem \ref{main}, we have $(\tilde{u}_j)^1_k\to 0$ when $q\to 0$. On the A-model side, since $q=0$, the $S-$matrix $(\mathring{S}^\alpha_{\, \ \beta}(z))$ is diagonal. Therefore, we also have $(u_j)^1_k\to 0$ when $q\to 0$ under the identification in Theorem \ref{main}. This means that in the localization graph of the equivariant GW invariants of $\bP^1$, we can only have a constant map to $p_2\in\bP^1$. Since $H|_{p_2}=\w_2=0$ and $t^0=0$, we can not have any primary insertions.
Therefore, in the large radius limit, we get
\begin{eqnarray*}
F^{\bP^1,\bC^*}_{g,n}(\bu_1,\cdots,\bu_n;\bt) &=& \sum_{a_1,\ldots, a_n\in \bZ_{\geq 0}}
\int_{[\Mbar_{g,n}(\bP^1,0)]^\vir}
\prod_{j=1}^n \ev_j^*((u_j)^2_{a_j}\phi_2(0)) \psi_j^{a_j}\\
&=&\sum_{a_1,\ldots, a_n\in \bZ_{\geq 0}}
\frac{1}{-\w_1}\int_{\Mbar_{g,n}}\prod_{j=1}^n (u_j)^2_{a_j} \psi_j^{a_j}\Lambda_{g}^{\vee}(-\w_1),
\end{eqnarray*}
where
$$
\Lambda_g^\vee(u)=u^g-\lambda_1 u^{g-1}+\cdots+(-1)^g\lambda_g.
$$
and $\lambda_j=c_j(\bE)$ is the $j$-th Chern class of the Hodge bundle. At the same time, we also have $\mathring{S}^2_{\, \ 2}=(\hat{\phi}_2(0),\hat{\phi}_2(0))=1$. So $\frac{(u_j)^2_k}{\sqrt{-\w_1}}=(\tilde{u}_j)^2_k$. Therefore Theorem \ref{main} specializes to
$$\omega_{g,n}|_{\frac{1}{\sqrt{-2}}W^2_k(Y_j)=
\frac{(u_j)^2_k}{\sqrt{-\w_1}}}=(-1)^{g-1+n}\sum_{a_1,\ldots, a_n\in \bZ_{\geq 0}}
\frac{1}{-\w_1}\int_{\Mbar_{g,n}}\prod_{j=1}^n (u_j)^2_{a_j} \psi_j^{a_j}\Lambda_{g}^{\vee}(-\w_1).$$

Now we study the expansion of $\xi_{2,0}$ near the point $Y=0$ in the coordinate $Z=e^{\frac{x}{\w_1}}$.
We have
$$
\xi_{2,0}=\frac{1}{\sqrt{-1}}\sqrt{\frac{2}{-\w_1}}\frac{-\w_1}{Y+\w_1}.
$$
Since $Z=Ye^{\frac{Y}{\w_1}}$, by taking the differential we have
$$\frac{dZ}{Z}=\frac{Y+\w_1}{Y\w_1}dY.$$
Therefore, $\xi_{2,0}=-\frac{1}{\sqrt{-1}}\sqrt{\frac{2}{-\w_1}}\frac{dY}{\frac{dZ}{Z}}\frac{1}{Y}$.
Let
$$\xi_{2,0}=\sum_{\mu=0}^{\infty}C_\mu Z^\mu.$$
near the point $Y=0$. Then we have
$$
C_\mu =\Res_{Y\to 0}\xi_{2,0}Z^{-\mu}\frac{dZ}{Z}\\
=-\frac{1}{\sqrt{-1}}\sqrt{\frac{2}{-\w_1}}\Res_{Y\to 0} e^{-\frac{\mu Y}{\w_1}}\frac{dY}{Y^{\mu+1}}\\
=-\frac{1}{\sqrt{-1}}\sqrt{\frac{2}{-\w_1}}\frac{(-\frac{\mu}{\w_1})^{\mu}}{\mu!}.
$$
Therefore
$$
W^2_k=-\frac{1}{\sqrt{-1}}\sqrt{\frac{2}{-\w_1}}\w_1\sum_{\mu=0}^{\infty}\frac{(-\frac{\mu}{\w_1})^{\mu}}{\mu!}
(-\frac{\mu}{\w_1})^{k+1}Z^{\mu-1}dZ.
$$

On A-model side, let
$$
(u_j)^2_{a_j}=\sum_{\mu_j=0}^{\infty}\frac{(-\frac{\mu_j}{\w_1})^{\mu_j}}{\mu_j!}
(\frac{\mu_j}{\w_1})^{a_j}Z^{\mu_j}_j.
$$
Then
\begin{eqnarray*}
&& F^{\bC,\bC^*}_{g,n}(\bu_1,\cdots,\bu_n)\\
&=&  \sum_{a_1,\ldots, a_n\in \bZ_{\geq 0}}
\frac{1}{-\w_1}\int_{\Mbar_{g,n}}\prod_{j=1}^n  \psi_j^{a_j}\Lambda_{g}^{\vee}(-\w_1)
\prod_{j=1}^{n}(\sum_{\mu_j=0}^{\infty}\frac{(-\frac{\mu_j}{\w_1})^{\mu_j}}{\mu_j!}
(-\frac{\mu_j}{\w_1})^{a_j}Z^{\mu_j}_j)\\
&=& \sum_{a_1,\ldots, a_n\in \bZ_{\geq 0}}
\frac{1}{-\w_1}\int_{\Mbar_{g,n}}\prod_{j=1}^n  (-\frac{\mu_j\psi_j}{\w_1})^{a_j}\Lambda_{g}^{\vee}(-\w_1)
\prod_{j=1}^{n}(\sum_{\mu_j=0}^{\infty}\frac{(-\frac{\mu_j}{\w_1})^{\mu_j}}{\mu_j!}Z^{\mu_j}_j).
\end{eqnarray*}
By the ELSV formula \cite{ELSV, GV},
\begin{eqnarray*}
H_{g,\mu} &=& \frac{(2g-2+|\mu|+n)!}{|\Aut(\mu)|}
\prod_{j=1}^n \frac{\mu_j^{\mu_j}}{\mu_j!}
\int_{\Mbar_{g,n}}\frac{\Lambda_g^\vee(1)}{\prod_{j=1}^n(1-\mu_j)}\\
&=&\frac{(2g-2+|\mu|+n)!}{|\Aut(\mu)|}
\prod_{j=1}^n \frac{\mu_j^{\mu_j}}{\mu_j!}
\int_{\Mbar_{g,n}}\frac{\Lambda_g^\vee(-\w_1)(-\w_1)^{2g-3+2n}}{\prod_{j=1}^n(-\w_1-\mu_j)}.\\
\end{eqnarray*}

So 
$$
F^{\bC,\bC^*} = \sum_{\ell(\mu)=n}
\frac{|\Aut(\mu)|}{(2g-2+|\mu|+n)!(-\w_1)^{2g-2+|\mu|+n} } H_{g,\mu} \sum_{\sigma\in S_n} \prod_{j=1}^n Z^{\mu_j}_{\sigma(j)}.
$$
When $\w_1=-1$, this is just the generating function of Hurwitz numbers.

Let $W_{g,n}(Z_1,\cdots,Z_n)$ be the expansion of $\omega_{g,n}(Y_1,\cdots,Y_n)$ in the coordinate $Z$ near $Y=0$. Then we have

\begin{corollary}[Bouchard-Mari\~{n}o conjecture]
For $n>0$ and $2g-2+n>0$, the invariant $W_{g,n}(Z_1,\cdots,Z_n)$ for the curve $x=Y+\w_1\log Y$ satisfies
\begin{eqnarray*}
&& \int^{Z_1}_0\cdots\int^{Z_n}_0W_{g,n}(Z_1,\cdots,Z_n)\\
&=& (-1)^{g-1+n}\sum_{a_1,\ldots, a_n\in \bZ_{\geq 0}}
\frac{1}{-\w_1}\int_{\Mbar_{g,n}}\prod_{j=1}^n  \psi_j^{a_j}\Lambda_{g}^{\vee}(-\w_1)
\prod_{j=1}^{n}(\sum_{\mu_j=0}^{\infty}\frac{(-\frac{\mu_j}{\w_1})^{\mu_j+a_j}}{\mu_j!}
Z^{\mu_j}_j)\\
&=& (-1)^{g-1+n}\sum_{\ell(\mu)=n}
\frac{|\Aut(\mu)|H_{g,\mu}}{(2g-2+|\mu|+n)!(-\w_1)^{2g-2+|\mu|+n} } \sum_{\sigma\in S_n} \prod_{j=1}^n Z^{\mu_j}_{\sigma(j)}.
\end{eqnarray*}
In particular, when $\w_1=-1$, the right hand side is the generating function of Hurwitz numbers and the Bouchard-Mari\~{n}o conjecture is recovered.
\end{corollary}

\begin{appendix}

\section{Bessel functions} \label{sec:bessel}
In this section, we give a brief review of Bessel functions.

The Bessel's differential equation is
\begin{equation} \label{eqn:bessel}
x^2\frac{d^2 y}{dx^2} + x\frac{dy}{dx} + (x^2-\alpha^2) y=0.
\end{equation}
The {\em Bessel function of the first kind} is defined by
$$
J_\alpha(x) = \sum_{m=0}^\infty \frac{(-1)^m}{m!\Gamma(m+\alpha+1)}(\frac{x}{2})^{2m+\alpha}.
$$
The {\em Bessel function of the second kind} is defined by
$$
Y_\alpha(x)=\frac{J_\alpha(x)\cos(\alpha\pi)-J_{-\alpha}(x)}{\sin(\alpha \pi)}.
$$
When $n$ is an integer, $Y_n(x):= \lim_{\alpha\to n}Y_\alpha(x)$.

$J_\alpha(x)$ and $Y_\alpha(x)$ form a basis of the
2-dimensional space of solutions to the Bessel's differential equation \eqref{eqn:bessel}.

Replacing $x$ by $ix$ in \eqref{eqn:bessel}, one obtains the
the modified Bessel differential equation
\begin{equation}\label{eqn:modified-bessel}
x^2\frac{d^2 y}{dx^2} + x\frac{dy}{dx} - (x^2 + \alpha^2) y=0.
\end{equation}

The {\em modified Bessel function of the first kind} is defined by
$$
I_\alpha(x)= i^{-\alpha} J_\alpha(ix) =
\sum_{m=0}^\infty \frac{1}{m!\Gamma(m+\alpha+1)}(\frac{x}{2})^{2m+\alpha}.
$$
The {\em modified Bessel function of the second kind} is defined by
$$
K_\alpha(x) = \frac{\pi}{2}\frac{I_{-\alpha}(x)-I_\alpha(x)}{\sin(\alpha \pi)}.
$$

The following integral formulas are valid when $\Re(x)>0$:
\begin{eqnarray*}
I_\alpha(x) &=&  \frac{1}{\pi} \int_0^\pi e^{x\cos\theta}\cos(\alpha\theta) d\theta
-\frac{\sin(\alpha\pi)}{\pi}\int_0^\infty e^{-x\cosh t -\alpha t} dt\\
K_\alpha(x) &=& \int_0^\infty e^{-x\cosh t} \cosh(\alpha t) dt = \frac{1}{2} \int_{t\in \gamma_{0,0}}e^{-x\cosh t - \alpha t} dt
\end{eqnarray*}
where $\gamma_{0,0}$ is the real line with the standard orientation:
\begin{figure}[h]
\begin{center}
\psfrag{infty}{\small $+\infty$}
\psfrag{-infty}{\small $-\infty$}
\includegraphics[scale=0.6]{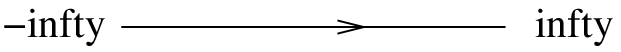}
\caption{The contour $\gamma_{0,0}$}
\end{center}
\end{figure}

\begin{eqnarray*}
&& e^{\alpha \pi i} K_\alpha(x) + i\pi I_\alpha (x)
=  \frac{\pi}{2} \frac{e^{\alpha \pi i}I_{-\alpha}(x) - e^{-\alpha \pi i} I_\alpha(x)}{\sin (\alpha \pi)}\\
&=& \frac{e^{\alpha \pi i}}{2}\int_{-\infty}^0 e^{-x\cosh t-\alpha t}dt
+\frac{e^{\alpha \pi i }}{2} \int_0^{2\pi} e^{-x\cos(i\theta)-\alpha(i\theta)} d(i\theta)
+\frac{e^{-\alpha \pi i}}{2}\int_0^\infty e^{-x\cosh t -\alpha t}dt \\
&=& \frac{e^{\alpha \pi i}}{2}\int_{\gamma_{0,1}} e^{-x\cosh t-\alpha t} dt
\end{eqnarray*}
where $\gamma_{0,1}$ is the following contour:
\begin{figure}[h]
\begin{center}
\psfrag{infty}{\small $+\infty +2\pi i$}
\psfrag{-infty}{\small $-\infty$}
\psfrag{0}{\small $0$}
\psfrag{2pi}{\small $2\pi i$}
\includegraphics[scale=0.6]{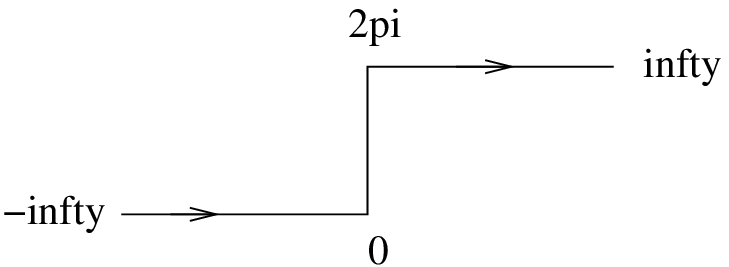}
\caption{The contour $\gamma_{0,1}$}
\end{center}
\end{figure}

Therefore,
\begin{equation}\label{eqn:gamma-zero}
\int_{\gamma_{0,0}}e^{-x\cosh t-\alpha t} dt = \frac{\pi}{\sin(\alpha \pi)}(I_{-\alpha}(x)-I_\alpha(x))
\end{equation}
\begin{equation}\label{eqn:gamma-one}
\int_{\gamma_{0,1}}e^{-x\cosh t -\alpha t} dt = \frac{\pi}{\sin(\alpha \pi)}(I_{-\alpha}(x)-e^{-2\alpha \pi i} I_\alpha(x) )
\end{equation}
For any integers $l_1,l_2$ with $l_1+l_2 \geq 0$, let $\gamma_{l_1,l_2}$ be the following contour:
\begin{figure}[h]
\begin{center}
\psfrag{infty}{\small $+\infty + 2 l_2 \pi i $}
\psfrag{-infty}{\small $-\infty - 2l_1 \pi i$}
\psfrag{L1}{\small $-2 l_1 \pi i $}
\psfrag{L2}{\small $2 l_2 \pi i$}
\includegraphics[scale=0.6]{l.eps}
\caption{The contour $\gamma_{l_1,l_2}$}
\end{center}
\end{figure}

\begin{lemma}\label{integral-bessel}
For any $l_1,l_2 \in \bZ$ such that $l_1+l_2\geq 0$, we have
\begin{equation}\label{eqn:gamma-l}
\int_{\gamma_{l_1,l_2}}e^{-x\cosh t -\alpha t} dt =
\frac{\pi}{\sin(\alpha\pi)} \big(e^{2 l_1 \alpha \pi i} I_{-\alpha}(x)
-e^{-2 l_2 \alpha \pi i} I_\alpha(x)\big)
\end{equation}
\end{lemma}
\begin{proof} We observe that
\begin{equation} \label{eqn:shift}
\int_{\gamma_{l_1- k ,l_2 + k}} e^{-x\cosh t-\alpha t} dt = e^{-2k \alpha \pi i}
\int_{\gamma_{l_1,l_2}} e^{-x\cosh t-\alpha t} dt.
\end{equation}
In particular,
$$
\int_{\gamma_{l_1 ,-l_1}} e^{-x\cosh t-\alpha t} dt = e^{2\ell_1 \alpha \pi i} \int_{\gamma_{0, 0}} e^{-x\cosh t-\alpha t} dt
= \frac{\pi}{\sin(\alpha \pi)}(e^{-2 l_1 \alpha \pi i} I_{-\alpha}(x)- e^{2 l_1\alpha \pi i} I_\alpha(x))
$$
This proves \eqref{eqn:gamma-l} in case $l_1+ l_2=0$. If $l_1+ l_2>0$ then
\begin{equation}\label{eqn:decompose}
\gamma_{l_1,l_2}= \sum_{k=-l_1}^{l_2-1}\gamma_{1-k,k}-\sum_{k=1-l_1}^{l_2-1}\gamma_{-k,k}.
\end{equation}
Equations \eqref{eqn:shift} and \eqref{eqn:decompose} imply
\begin{eqnarray*}
&& \int_{\gamma_{l_1,l_2}}e^{-x\cosh t -\alpha t} dt \\
& = & \Big(\sum_{k =-l_1}^{l_2-1} e^{-2 k\alpha \pi i}\Big)\cdot \int_{\gamma_{0,1}}e^{-x\cosh t -\alpha t} dt
-  \Big(\sum_{k =1-l_1}^{l_2-1} e^{-2 k\alpha \pi i}\Big)\cdot \int_{\gamma_{0,0}}e^{-x\cosh t -\alpha t} dt
\end{eqnarray*}
Equation \eqref{eqn:gamma-l} follows from the above equation and \eqref{eqn:gamma-zero}, \eqref{eqn:gamma-one}.
\end{proof}

\section{The Equivariant Quantum Differential Equation for $\bP^1$} \label{sec:QDE}
The equivariant quantum differential equation of $\bP^1$ is the vector equation
$$
zq\frac{d}{dq}\vec{I}= \left(\begin{array}{cc} 0 & q-\w_1\w_2\\ 1 & \w_1+\w_2\end{array} \right)\vec{I}
$$
which is equivalent to the following scalar equation:
\begin{equation} \label{eqn:qde}
(z q\frac{d}{dq}-\w_1)(zq\frac{d}{dq}-\w_2)I = q I.
\end{equation}
Let
$$
I=e^{\frac{\w_1+\w_2}{2z} \log q} y,\quad  x=\frac{2\sqrt{q}}{z}.
$$
Then \eqref{eqn:qde} is equivalent to
$$
x^2\frac{d^2y}{dx^2} +x\frac{dy}{dx} - (x^2 + (\frac{\w_1-\w_2}{2z})^2) y =0
$$
which is the modified Bessel differential equation \eqref{eqn:modified-bessel}
with $\alpha=\frac{\w_1-\w_2}{2z}$. When $\w_1-\w_2\neq 0$, any solution to \eqref{eqn:qde} is of the form
$$
I = e^{\frac{\w_1+\w_2}{2z}\log q }\Big(c_1 I_{\frac{\chi^1}{z}}(\frac{2\sqrt{q}}{z}) +
c_2 I_{\frac{\chi^2}{z}}(\frac{2\sqrt{q}}{z}) \Big).
$$
where $\chi^1=\w_1-\w_2 = -\chi^2$, and $c_1, c_2$ are functions of $\w_1,\w_2, z$.

\end{appendix}

\end{document}